\documentclass[11pt,reqno]{amsart}

\usepackage[colorlinks=true, linkcolor=blue]{hyperref}
\numberwithin{equation}{section}
\usepackage{times}
\usepackage{amsmath,amsfonts,amstext,amssymb,amsbsy,amsopn,amsthm,eucal}
\usepackage{txfonts}
\usepackage{dsfont}
\usepackage{graphicx}   
\usepackage{hyperref}
\usepackage{accents}
\usepackage{enumerate}
\usepackage{xcolor}
\usepackage{verbatim}
\usepackage{esint}

\usepackage[normalem]{ulem}
\usepackage{cancel}

\newcommand{\Vol}{{\rm Vol}}

\newtheorem{theorem}{Theorem}[section]

\newtheorem{proposition}[theorem]{Proposition}
\newtheorem{lemma}[theorem]{Lemma}
\newtheorem{corollary}[theorem]{Corollary}
\theoremstyle{definition}
\newtheorem{definition}[theorem]{Definition}

\theoremstyle{remark}
\newtheorem{remark}[theorem]{Remark}
\theoremstyle{remark}

\theoremstyle{remark}

\theoremstyle{remark}
\theoremstyle{remark}

\title{Geometric sets of elliptic equations with Dini coefficients}

\author{Wenshuai Jiang}
\address[Wenshuai Jiang]{School of Mathematical Sciences, Zhejiang University, Hangzhou, 310058, China}
 \email{wsjiang@zju.edu.cn}

\author{Junyuan Wang}
\address[Junyuan Wang]{School of Mathematical Sciences, Zhejiang University, Hangzhou, 310058, China}
\email{wangjunyuan@zju.edu.cn}

\date{}

\begin{document}

\maketitle
\vspace{-10mm}

\begin{abstract}
    We study divergence-form elliptic equations with Dini coefficients, extending the frameworks of \cite{NV} for Lipschitz coefficients and \cite{HJ1} for H\"older coefficients. We first establish an almost monotonicity property for the doubling index under merely continuous coefficients. This leads to weak volume estimates for critical sets and uniform volume estimates for nodal sets. The Dini assumption is essential here since we construct an example with continuous coefficients and bounded doubling index whose nodal set has infinite measure. We also prove polynomial growth estimates for sub-level sets. Using a different approach, we establish quantitative uniqueness of tangent maps and a cone-splitting principle, which yield explicit Minkowski-type estimates for critical sets. Finally, combined with the results of Kenig--Zhao \cite{KZ4}, our estimates give measure bounds for boundary critical sets.
\end{abstract}

\vspace{-8mm}

\tableofcontents

\section{Introduction}
\indent In this paper, we consider the elliptic partial differential equations in divergence form and study the geometric sets of the solutions, including the nodal sets, sub-level sets and critical sets. Let $D$ be a connected domain and $u:D\subset\mathbb{R}^{n}\to\mathbb{R}$ be a $C^{1}$-function, then we denote
\begin{equation}
    Z(u)=\{x:u(x)=0\}\; \text{and} \;C(u)=\{x:|\nabla u(x)|=0\}.
\end{equation}
as the nodal set and critical set of $u$ respectively. The set $S(u)=Z(u)\cap C(u)$ is called the singular set of $u$.\\
\indent  A central conjecture in this area is the well-known Yau's conjecture in \cite{Yau}. In \cite{DF1}, Donnelly and Fefferman gave a positive answer to Yau's conjecture for real analytic metric. Later, Lin \cite{Linconj} used frequency arguments to give a different proof of the sharp upper bound, see also Arya-Banerjee-Garofalo \cite{ABG}. In non-anaytic settings, Hardt and Simon \cite{HS} obtained an explicit upper bound estimate for smooth metric, this upper bound was also proved by Han and Lin \cite{HL} using a different method. There are many other important progress about Yau's conjecture, we refer the reader to Br\"{u}ning \cite{B}, Colding-Minicozzi \cite{CM}, Dong \cite{D}, Nadirashvili \cite{Nadi} and Sogge-Zelditch \cite{SZ} for relevant study. Recently, a major breakthrough was obtained by Logunov. For smooth closed manifolds, he established a sharp lower bound for Yau's conjecture \cite{Loupper} and obtained a polynomial upper bound \cite{Lolower}. \\
\indent To get the estimate on nodal set, weaker regularity of the coefficients is good enough.  Under the doubling assumption, Hardt and Simon \cite{HS} proved the $\mathcal{H}^{n-1}$-measure estimates for nodal sets of elliptic PDEs in non-divergence form with continuous coefficients and in divergence form with Hölder coefficients. Later, Han and Lin \cite{HLparabolic} generalized this result to parabolic equations with Hölder coefficients. For much weaker coefficients, Kim \cite{Kim} studied the broken quasilinear elliptic equation with only Dini continuous coefficients, and obtained similar estimates for nodal sets. More recently, Huang and Jiang \cite{HJ2} considered parabolic equations with Lipschitz and H\"{o}lder coefficients, they proved the $\mathcal{H}^{n-1}$-measure estimates for nodal sets at any time slice. For more related study, see \cite{LZ22,Lo21,Lo24,TY}.\\
\indent The geometry of nodal sets is closely related to that of singular and critical sets, and is strongly influenced by the way in which the underlying functions degenerate. This naturally motivates the study of critical sets. On the other hand, the study of singular sets and critical sets of solutions to elliptic and parabolic equations is a fundamental problem connecting quantitative unique continuation, geometric measure theory, and the geometric of level sets. For any $C^{1}$-function $u$, the classical implicit theorem tells us that $Z(u)$ is locally a $C^{1}$-graph away from $S(u)$. Therefore, the singular set provides information on the degeneracy of the nodal set, whereas the critical set captures degeneracy of all level sets and is intrinsic under translation $u\mapsto u+c$. In \cite{Hansingular}, Han proved the structure theorem for singular sets of elliptic equations, and he obtained the $(n-2)$-rectifiability of singular sets. Furthermore, a central objective is to quantitatively relate the geometry of singular and critical sets to the growth of solutions, measured by frequency functions or doubling indices. The main conjecture in this area goes back to Lin \cite{Linconj}, he predicted that for any non-trivial solution $u$ to the elliptic equation with Lipschitz coefficients, the following Hausdorff measure estimate holds.
\begin{equation}
    \mathcal{H}^{n-2}(S(u)\cap B_{1/2}(0))\leq CN(0,1)^{2},
\end{equation}
where $N(x,r)$ denotes the frequency function on $B_{r}(x)$ and $C$ is a universal constant. \\
\indent We review some important developments in this area.  In \cite{98singular}, Han, Hardt and Lin proved the $\mathcal{H}^{n-2}$-finiteness of singular sets of the solutions to elliptic equations with smooth coefficients. Later, the $\mathcal{H}^{n-2}$-finiteness of critical sets is proved by Hardt, Hoffmann-Ostenhof, M, Hoffmann-Ostenhof, T and Nadirashvili \cite{99critical} under the same setting. For non-smooth coefficients, Cheeger, Naber and Valtorta \cite{CNV} introduced the quantitative stratification technique to show the weak volume estimates for critical sets of elliptic equations with Lipschitz coefficients. Later, Naber and Valtorta \cite{NV} proved the Minkowski-type volume estimates for critical sets under Lipschitz assumptions. Recently, Kim  \cite{Kim} generalized the structure theorem to quasilinear elliptic equations with Dini coefficients. Kenig and Zhao \cite{KZ1} obtained the volume estimates for singular sets of harmonic functions in $C^{1}$-Dini domain. More recently, Huang and Jiang \cite{HJ1} generalized the results in \cite{NV} to Hölder coefficients under the doubling assumption. In \cite{HKM} , Hallgrent, Koirala and Ma proved the volume estimates for nodal sets and singular sets for parabolic equations with Lipschitz coefficients. In addition, there are many other study on singular sets and critical sets in various settings, including high-order elliptic equations, elliptic homogenization, Dirichlet eigenfunctions, and so on, we refer readers to \cite{HHL,LS24,Z21}.\\
\indent The study of boundary singular sets and critical sets is another story. The earliest results appear to originate from optimization and control theory. There is a classical problem as follows.\\ 
\indent \textbf{Problem:} Let $D$ be a connected domain and $0\leq u\leq 1$ be a harmonic function in $D$. Fix two points $x_{0},x_{1}\in D$, we wish to minimize $u(x_{1})$ among all such $u$ such that $u(x_{0})=1/2$. \\
\indent For $D$ is a smooth domain, in the 1970s, Schmidt and Weck \cite{SW}, \cite{Wec} proved that $\mathcal{H}^{n-1}(S(u)\cap \partial D)=0$ for the boundary singular set . As a result, they obtained \textit{The Bang-Bang Property}, i.e., there exists a unique minimizer $\tilde{u}$ to the above problem and $\tilde{u}|_{\partial D}=1_{E}$ for some $E\subset \partial D$. This result was generalized to $C^{1,1}$ domain by Lin in 1991, and the Theorem 2.3 in Lin \cite{Linconj} tells us ${\rm dim}_{H}(S(u)\cap \partial D)\leq n-2$. Besides, in the 1990s, Adolfsson, Escauriaza, Kenig \cite{AEK} and Kenig-Wang \cite{KW} showed that $\mathcal{H}^{n-1}(S(u)\cap \partial D)=0$ if $D$ is a convex domain. In 1997, Adolfsson, Escauriaza \cite{AE} proved that ${\rm dim}_{H}(S(u)\cap \partial D)\leq n-2$ for $C^{1}$-Dini domains. Recently, Tolsa \cite{Tol} showed that $\mathcal{H}^{n-1}(S(u)\cap \partial D)=0$ if $D$ is a $C^{1}$ domain and  $u\not\equiv 0$ is a harmonic function in $D$. \\
\indent  In 2022, Kenig and Zhao \cite{KZ1} proved the uniform finiteness of singular sets of harmonic functions in $C^{1}$-Dini domains. Furthermore, in \cite{KZ3}, they have constructed an infinite singular set of some harmonic function in non-Dini domain. Therefore, the above result for $C^{1}$-Dini domain is sharp. Besides, in \cite{KZ4}, Kenig and Zhao proved the uniform finiteness of critical sets of harmonic functions in $C^{1,\alpha}$ domains, and they also mentioned that {\textit{there are serious difficulties in extending the analysis of the singular set in \cite{KZ1} to the critical set.}} Since $C^{1}$-Dini domain is the sharp domain, it is also an important \textbf{motivation} for our study of the boundary critical sets of harmonic functions in it.\\
\indent Besides, a natural question is what coefficient regularity is necessary for quantitative measure estimates of the above geometric sets. To our knowledge, when considering the elliptic equations in divergence form, the weakest condition to ensure the $C^{1}$-regularity of the solutions is the Dini mean oscillation condition, due to \cite{DMO}. Therefore, the singular and critical sets can be well-defined in this case. \\
\indent In this paper, we consider the equation
\begin{equation}\label{elliptic equation}
    {\rm div} (A\nabla u)=\partial_{i}(a^{ij}(x)\partial_{j}u)=0 
\end{equation}
in $B_{4}(0)\subset \mathbb{R}^{n}$ with uniform ellipticity condition
\begin{equation}\label{unifomly elliptic}
    (1+\lambda)^{-1} I\leq A(x)\leq (1+\lambda)I.
\end{equation}
Denote the mudulus of continuity of $A$ by $\theta$, i.e., 
\begin{equation}\label{continuous coefficient}
     \sup_{x,y\in B_{4}}|A(x)-A(y)|:= \theta(|x-y|),
\end{equation}
where $|A|:=\sqrt{\sum_{i,j=1}^{n}|a^{ij}|^{2}}$.  We say the coefficient matrix $A$ is \textit{$\omega$-Dini continuous} in $B_{4}(0)$ if $\theta(t)\leq \omega(t)$ for all $0\leq t\leq 8$ and
 \begin{equation}
      \omega:[0,8]\to \mathbb{R}_{+}\;\text{with}\;\lim_{r\to 0}\omega(r)=0\;\text{and}\;\int_{0}^{8}\frac{\omega(t)}{t}dt<\infty.
 \end{equation}

For any given ball $B$, we set
\begin{equation}
    \bar{A}_{B}=\fint_{B}A:=\frac{1}{|B|}\int_{B}A.
\end{equation}
And we say $A$ satisfies the \textit{Dini mean oscillation condition} (DMO) if 
\begin{equation}\label{DMO}
    \int_{0}^{1}\frac{\omega(r)}{r}dr<\infty\;\text{where}\;\omega_{A}(r)={\rm sup}_{x\in B_{4}}\fint_{B_{r}(x)}|A-\bar{A}_{B_{r}(x)}|\leq \omega(r).
\end{equation}
It is easy to see that Dini continuous condition implies DMO. Besides, the average matrix sequence $\{\bar{A}_{B_{2^{-j}r}(x)}\}_{j\geq 0}$ is a Cauchy sequence under the above DMO condition, then we denote $  A_{c}:=\lim_{j\to \infty}\bar{A}_{B_{2^{-j}r}(x)}$. By Lebesgue differentiation theorem, $A_{c}=A$ a.e., and one can check that $A_{c}$ is a continuous representative of $A$ with  modulus of continuity $\theta(r)=C(n)\int_{0}^{10r}\frac{\omega(t)}{t}dt$.\\ 
\indent However, since the unique continuation fails even for the elliptic equations with Hölder continuous coefficients. Therefore, to exclude the case that $u$ vanishes in an open set, we may assume the solution $u$ satisfies the \textit{doubling condition}, i.e., there exists $\Lambda>0$ such that
\begin{equation}\label{doubling assumption}
    {\rm log}_{4}\frac{\fint_{B_{2r}(x)}u^{2}}{\fint_{B_{r}(x)}u^{2}}\leq \Lambda
\end{equation}
for any $B_{2r}(x)\subseteq B_{4}(0)$. Equivalently, by interior estimate, one can also assume $\Lambda$ is large enough such that
\begin{equation}
    \log_{2}\frac{||u||_{L^{\infty}(B_{2r}(x))}}{||u||_{L^{\infty}(B_{r}(x))}}\leq  {\rm log}_{4}\frac{C(n,\lambda,\theta)\fint_{B_{4r}(x)}u^{2}}{\fint_{B_{r}(x)}u^{2}}\leq C+2\Lambda\leq 3\Lambda.
\end{equation}
What's more, if we study the critical sets, we will assume the following doubling condition
\begin{equation}\label{doubling for critical set}
    \sup_{B_{2r}(x)\subseteq B_4(0)} {\rm log}_{4}\frac{\fint_{B_{2r}(x)}|u-u(x)|^{2}}{\fint_{B_{r}(x)}|u-u(x)|^{2}}\leq \Lambda.
\end{equation}

\subsection{Nodal sets and sub-level sets}
\indent In this subsection, we consider the nodal sets of elliptic equations in divergence form with only DMO coefficients. The first main theorem in this part is the following uniform Minkowski estimate for nodal set.
\begin{theorem}\label{nodal set main theorem}
      Let $u\in W^{1,2}(B_{4}(0))$ be a nonzero weak solution to \eqref{elliptic equation}, with the coefficients satisfying \eqref{unifomly elliptic}, \eqref{continuous coefficient} and \eqref{DMO} in $B_{4}(0)\subset\mathbb{R}^{n}$. Assume that $u$ satisfies the doubling condition \eqref{doubling assumption}, then for any $0<r<1$, we have
     \begin{equation}
           {\rm Vol}\Big(B_{r}\big(Z(u)\big)\cap B_{1}(0)\Big)\leq C(n,\lambda,\theta,\omega,\Lambda)r.
     \end{equation}
Therefore $\mathcal{H}^{n-1}(Z(u)\cap B_{1}(0))\leq C(n,\lambda,\theta,\omega,\Lambda)<\infty.$
\end{theorem}
\begin{remark}
When considering the measure estimates for nodal sets, in some sense, DMO condition is optimal since $C^{1}$-regularity of the solution is sharp. Actually, if $A$ is continuous, then $u$ only has $C^{\alpha}$-regularity, even if we assume that $u\not \equiv 0$ satisfies doubling assumption, a classical example from quasiconformal geometry will give a counterexample in $\mathbb{R}^{2}$, i.e., the nodal set of $u$ has infinite length, see Appendix A.
\end{remark}
\indent In general, when $A$ has Lipschitz regularity, there is an essential tool  to study the nodal sets and singular sets of the elliptic equations, which is called the Almgren's frequency function. In our case, the natural monotonicity quantity is missing. With the doubling assumption, we will establish an almost monotonicity for doubling index only under continuous coefficients. Combing the quantitative stratification technique, inductive covering arguments and an $\epsilon$-regularity result, we can prove the weak volume estimates similar to \cite{CNV}. Based on the weak volume estimates for singular strata and critical sets, we can obtain the desired Minkowski volume estimate.\\
\indent Next, we consider another related geometric set, the sub-level set of $u$. Given $\epsilon>0$, we denote 
\begin{equation}
    \mathcal{S}_{\epsilon}:=\{x:|u(x)|\leq \epsilon\}.
\end{equation}
For elliptic equation ${\rm div}(A(x)\nabla u)=0$ with $A$ is Lipschitz continuous, and assume the normalization condition of $u$, Logunov and Malinnikova \cite{LM1} used the optimal upper and lower bound of nodal sets to obtain the polynomial growth of the volume of sub-level set, i.e., ${\rm Vol}(\mathcal{S}_{\epsilon})\leq C\epsilon^{\alpha}$, where $C=C(n,\lambda)$ and $\alpha=C(n,\lambda)/\Lambda$. \\
\indent As an application of Theorem \ref{nodal set main theorem}, one can generalize the above result to elliptic equations with only DMO coefficients under the doubling assumption.
\begin{theorem}\label{sub-level set theorem}
    Under the same setting as in Theorem \ref{nodal set main theorem}. Assume that $\fint_{B_{1}(0)}u^{2}=1$, then there exist positive constants $C(n,\lambda,\Lambda,\theta,\omega)$ and $\alpha=3/4\Lambda$ such that 
      \begin{equation}
          {\rm Vol}(\mathcal{S}_{\epsilon}\cap B_{1}(0))\leq C\epsilon^{\alpha}
      \end{equation}
     holds for any $0<\epsilon<1$. Especially, for any measurable set $E\subset B_{1}(0)$ with $|E|>0$, we have the Remez-type inequality
\begin{equation}
    \frac{{\rm sup}_{B_{1}(0)}|u|}{{\rm sup}_{E}|u|}\leq C(n,\lambda,\Lambda,\theta,\omega)\cdot(\frac{|B_{1}|}{|E|})^{4\Lambda/3}.
\end{equation}
\end{theorem}
\begin{remark}\label{remark sub-level set}
    Denote $E_{b}=S_{e^{-b}}=\{x:|u(x)|\leq e^{-b}\}$, and we will show that $\Vol(E_{b}\cap B_{1}(0))\leq Ce^{-\alpha b}$. In fact, we only need to consider the case $b$ is sufficiently large by adjusting the constant $C$ correspondingly. Actually by adjusting the constants in our proof, we can obtain $\Vol(E_{b}\cap B_{1}(0))\leq C(n,\lambda,\theta,\omega,\Lambda,\delta)e^{-(1-\delta) b/\Lambda}$ for $\delta\in (0,1)$.
\end{remark}
   \indent The value of the second conclusion is that we only require $E$ to have positive measure. In particular, it applies to the regions with many connected components, each of which is a small ball, in   which case a direct doubling estimate would become worse.
\subsection{Interior critical sets}

\indent In this subsection, we consider the interior critical sets. We generalize the main result in \cite{HJ1} to Dini case, which is sharp in some sense. And it is worth mentioning that this is also an explicit volume estimate regarding doubling index. 
\begin{theorem}\label{interior critical theorem}
    Let $u$ be a solution to ${\rm div}(A\nabla u)=0$ in $B_{2}(0)\subset \mathbb{R}^{n}$, where $A$ is $\omega$-Dini continuous with \eqref{unifomly elliptic}. Assume that $u$ satisfies doubling condition \eqref{doubling for critical set}, then for any $0<r\leq 1$, we have the following volume estimate for critical set
    \begin{equation}
        {\rm Vol}\Big(B_{r}(C(u))\cap B_{1}(0)\Big)\leq C(n,\lambda,\omega)^{\Lambda^{2}}r_{0}^{-n}\cdot r^{2},
    \end{equation}
where $r_{0}$ satisfies $\Omega(\kappa r_{0})=C(n,\lambda,\omega)^{-\Lambda^{2}}$ with $\kappa=100(1+\lambda)$ and $\Omega(r):=\omega(r)+\int_{0}^{r}\frac{\omega(t)}{t}dt$. Especially, we have the Hausdorff estimate $\mathcal{H}^{n-2}(C(u)\cap B_{1}(0))\leq C(n,\lambda,\omega)^{\Lambda^{2}}r_{0}^{-n}$.
\end{theorem}
\begin{remark}
    Note that if $\omega(r)=r^{\alpha}$ for some $\alpha\in(0,1]$, the above theorem gives us the same explicit estimate as in \cite{HJ1}. Besides, with only minor modifications, our proof yields a similar Minkowski estimate in the case of $L^{2}$-DMO coefficients, i.e., $A$ satisfies ${\rm sup}_{x\in B_{4}}(\fint_{B_{r}(x)}|A-\fint_{B_{r}(x)}A|^{2})^{1/2}\leq \omega(r)$, where $\omega$ is Dini continuous.
\end{remark}
\indent When considering the critical sets of elliptic equations with Dini coefficients. Compared with the case of $\alpha$-Hölder continuous coefficients, we lack the $\alpha$-index, we can not directly show the quantitative uniqueness and cone splitting using the method in \cite{HJ1} and \cite{NV} . Inspired by the methods for dealing with the elliptic equations with Dini coefficients in \cite{DMO},  \cite{KZ2} and \cite{LS24}, we use the harmonic approximation and projection arguments to show the error terms are summable under the Dini setting. Furthermore, we prove the quantitative cone splitting, and based on this, we can construct a universal Lipschitz structure for local critical points, which is the key point for us to obtain the uniform finiteness of volume estimate. Finally, we use the powerful inductive covering arguments developed by \cite{NV} to obtain the explicit volume estimates for critical sets.
\subsection{Boundary critical sets} In this subsection, we present a uniform measure estimate for critical set of harmonic function in a $C^{1}$-Dini domain, this result generalizes the main theorem in \cite{KZ4} to an optimal domain.
\begin{theorem}\label{boundary critiical theorem}
    Let $D\subset \mathbb{R}^{n}$ be a $C^{1}$-Dini domain (see Definition \ref{def7.2}) with $0\in \partial D$. Assume that $u\not \equiv 0$ is a harmonic function in $D\cap B_{5R}(0)$, satisfying the boundary condition $u=0$ in $\partial D\cap B_{5R}(0)$. Then we have the Hausdorff estimate for critical set
    \begin{equation}
        \mathcal{H}^{n-2}(C(u)\cap \overline{D}\cap B_{R}(0))\leq C<\infty,
    \end{equation}
    where $C$ depends on $n,R$, the Dini parameter $\omega$ and the upper bound of frequency function on $B_{5R}(0)$.
\end{theorem}
\indent In \cite{KZ4}, Kenig and Zhao also pointed out that once the inner volume estimate for critical sets as in \cite{HJ1} holds for the elliptic equations with Dini coefficients, their main result also holds for $C^{1}$-Dini domains, i.e., Theorem \ref{boundary critiical theorem} is true. In Section 7.2, we will give a brief proof sketch of their theorem, and verify the above theorem. \\
\indent \textbf{Notation: } As in \cite{NV}, we denote the inner product as $ \langle f,g\rangle:=\fint_{\partial B_{1}}fg$ for any $f,g\in L^{2}(\partial B_{1})$, and we set $||f||:=\langle f,f\rangle^{1/2}$. For an $n\times n$ matrix $A=(a^{ij})$, we denote its norm as $|A|:=\sqrt{\sum_{i,j=1}^{n}|a^{ij}|^{2}}$. By $H_{m}$ we will denote the space of homogeneous harmonic polynomial (hhp) with degree $m$. By $\Pi_{m}$ we will denote the projection from $L^{2}(\partial B_{1})$ to $H_{m}$. By $u_{x,r}$ we will denote the rescaled map with normalization on $B_{1}$ as in \eqref{2.5}, and $\tilde{u}_{x,r}$ as in \eqref{2.6} is used for the rescaled map with normalization on $\partial B_{1}$. By $c,c_{0},c_{1},c_{2},...$ we will denote small constants and by $C,C_{0},C_{1},C_{2},...$ large constants. The constants may change line by line.\\\\
\noindent \textbf{Disclosure of AI tools.} We use GPT 5.6 Sol  to search some relevant references. In particular, the references \cite{DMO,KZ1,KZ2,Kim}, together with our conversations with it, inspiring us to prove Proposition \ref{projection difference}. Besides, we use GPT 5.6 Sol to search for a counterexample in the continuous coefficient case, and we simplify the original example and verify it.

\noindent
\textbf{Acknowledgements:}
J. Wang would like to thank Xiujin Chen and Yiqi Huang for helpful discussions.

\section{Preliminaries}
In this section, we collect some preliminaries. First, we introduce the $C^{1}$-regularity for elliptic equations with DMO coefficients. To our knowledge, this is the weakest condition for getting the $C^{1}$-regularity for $u$. 
\begin{theorem}[\cite{DMO}]\label{DMO regularity}
    Assume that $A$ satisfies \eqref{unifomly elliptic} and \eqref{DMO}. Let $u\in W^{1,2}(B_{4}(0))$ be a weak solution to 
    \begin{equation}
        {\rm div}(A\nabla u)={\rm div} F
    \end{equation}
    in $B_{4}(0)$ with $F$ satisfies \eqref{DMO}. Then $u\in C^{1}(\overline{B_{1}(0)})$ and the following gradient estimate holds
    \begin{equation}
        ||\nabla u||_{L^{\infty}(B_{1}(0))}\leq C(||\nabla u||_{L^{1}(B_{2}(0))}+\int_{0}^{1}\frac{\tilde{\omega}_{F}(s)}{s}ds).
    \end{equation}
    Where $C$ is a uniform constant, and $\tilde{\omega}_{F}$ is defined by $\omega_{F}$, satisfying $\int_{0}^{1}\frac{\tilde{\omega}_{F}(s)}{s}ds<\infty$.
\end{theorem}
Next, we introduce some basic notions. Since there is a natural inclusion relationship between two balls at the same center, sometimes, it brings us convenience. Therefore, we firstly consider the doubling index on ball. For $x\in B_{1}$, $r\leq 1/(1+\lambda)^{1/2}$, we define \textit{the doubling index on ball} $B_{r}(x)$ as 
\begin{equation}
    D_{b}^{u}(x,r):=  {\rm log}_{4}\frac{\fint_{ B_{2}(0)}(u(x+rA_{x}(y))-u(x))^{2}dy}{\fint_{ B_{1}(0)}(u(x+rA_{x}(y))-u(x))^{2}dy}
\end{equation}
where $A_{x}(y)=(\sqrt{A})^{ij}y_{i}e_{j}$, $\sqrt{A}^{ij}$ is the square root of the matrix $a^{ij}(x)$. However, when dealing with the quantitative uniqueness, it is more suitable to consider \textit{the doubling index on sphere}, that is
\begin{equation}
      D_{s}^{u}(x,r):=  {\rm log}_{4}\frac{\fint_{\partial B_{2}(0)}(u(x+rA_{x}(y))-u(x))^{2}dy}{\fint_{ \partial B_{1}(0)}(u(x+rA_{x}(y))-u(x))^{2}dy}
\end{equation}
\indent \textbf{We emphasize here that}, expect for individual cases distinguished by subscripts, the notion $D(x,r)$ refers to the doubling index on ball in Sections 3 and 4, and refers to the doubling index on sphere in Section 5 to Section 7. In Section 3, we prove some important properties of doubling index on ball, what's more, following \cite{HJ1}, all properties can be similarly proved for doubling index on sphere. Besides, when considering the nodal sets and singular sets, we do not minus the value of $u$ at center in the definition of doubling index and the following rescaled map.\\
\indent By doubling condition \eqref{doubling assumption} or \eqref{doubling for critical set} and a direct estimate, we have $D_{b}(x,r)\leq C(\lambda)\Lambda$. Next we define the rescaled map as
\begin{equation}\label{2.5}
    u_{x,r}(y):=\frac{u(x+rA_{x}(y))-u(x)}{(\fint_{ B_{1}(0)}|u(x+rA_{x}(y))-u(x)|^{2}dS_{y})^{1/2}}.
\end{equation}
And for doubling index on sphere, we define the rescaled map as
\begin{equation}\label{2.6}
    \tilde{u}_{x,r}(y):=\frac{u_{x,r}(y)}{||u_{x,r}||}:=\frac{u(x+rA_{x}(y))-u(x)}{(\fint_{ \partial B_{1}(0)}|u(x+rA_{x}(y))-u(x)|^{2}dS_{y})^{1/2}}.
\end{equation}
Now denote that $\bar{u}=u_{x,r}(y)$, then 
\begin{equation}
    D^{u}(x,rs)=D^{\bar{u}}(0,s):= {\rm log}_{4}\frac{\fint_{B_{2s}(0)}|\bar{u}|^{2}}{\fint_{B_{s}(0)}|\bar{u}|^{2}}.
\end{equation}
for $x\in B_{1}(0)$, $r,s\leq 1/(1+\lambda)^{1/2}$. And $\bar{u}$ satisfies that 
\begin{equation}
    \bar{\mathcal{L}}\bar{u}=\partial_{i}(\bar{a}^{ij}\partial_{j}\bar{u})=0,
\end{equation}
where $\bar{a}(y)=a(x)^{-1/2}a(x+rA_{x}(y))a(x)^{-1/2}$ and $a(x):=A(x)$. Therefore by a direct calculation, we have $\bar{a}(0)=Id$ and if we assume $A$ satisfies \eqref{continuous coefficient}, then 
\begin{equation}
      |\bar{a}^{ij}(y)-\bar{a}^{ij}(z)|\leq C(\lambda)\theta(\sqrt{1+\lambda}|y-z|)
\end{equation}
for any $y,z\in B_{1}(0)$. \\
\indent Next, we consider the most special case, i.e., $u$ is a harmonic function, and we introduce some monotonicity results. Let
\begin{equation}
    H(x,r)=\int_{\partial B_{r}(x)}|u-u(x)|^{2},\text{and }\beta(x,r)= {\rm log}_{4}\frac{\fint_{\partial B_{2r}(x)}|u-u(x)|^{2}}{\fint_{ \partial B_{r}(x)}|u-u(x)|^{2}}.
\end{equation}
It is well-known that $\beta(x,r)$ is monotone nondecreasing with respect to $r$ and it relates to the frequency function. Next, we denote
\begin{equation}
    D(x,r)= {\rm log}_{4}\frac{\fint_{B_{2r}(x)}|u-u(x)|^{2}}{\fint_{ B_{r}(x)}|u-u(x)|^{2}}:= {\rm log}_{4}\frac{E(x,2r)}{E(x,r)}.
\end{equation}
Actually, it is also monotone nondecreasing with respect to $r$ if $u$ is a harmonic function.
\begin{lemma}\label{monotonicity for harmonic}
    Let $w$ be a harmonic function in $B_{2}(0)$, then for any $B_{2r}(x)\subset B_{1}(0)$, $D(x,r)$ is monotone nondecreasing with respect to $r$.
\end{lemma}
\begin{proof}
    We may assume that $x=0$ and recall that 
    \begin{equation}
        ||f||:=\sqrt{\fint_{\partial B_{1}(0)}f^{2}(x)dx}.
    \end{equation}
    Then we can write 
    \begin{equation}
        w=\sum_{k=0}^{\infty}a_{k}P_{k},
    \end{equation}
    where $P_{k}$ is homogeneous harmonic polynomial of degree $k$ with $||P_{k}||=1$, $a_{k}=\langle w, P_{k}\rangle$. Therefore
    \begin{equation}
        H(r):=\int_{\partial B_{r}(0)}|w-w(0)|^{2}=n\alpha(n)r^{n-1}\cdot \sum_{k=1}^{\infty}a_{k}^{2}r^{2k}, 
    \end{equation}
    then we have
    \begin{equation}
        E(r)=\fint_{B_{r}(0)}|w-w(0)|^{2}=\sum_{k=1}^{\infty}\frac{n}{2k+n}a_{k}^{2}r^{2k}:=\sum_{k=1}^{\infty}b_{k,n}r^{2k}.
    \end{equation}
    Let 
    \begin{equation}
        g(s):= {\rm log}_{4}E(e^{s})= {\rm log}_{4}\sum_{k=1}^{\infty}b_{k,n}e^{2ks}, 
    \end{equation}
    then $D(r)=g({\rm log}\,2+{\rm log}\,r)-g({\rm log}\,r)$. And $D^{'}(r)=\frac{1}{r}[g^{\prime}({\rm log}\,2+{\rm log}\,r)-g^{\prime}({\rm log}\,r)]$, to show $D^{'}\geq 0$, it suffices to show $g^{\prime\prime}(s)\geq 0$ for all $s\in (-\infty,0)$. Since
    \begin{equation}
        g^{\prime}(s)=\frac{1}{{\rm log}\,4}\cdot\frac{\sum_{k=1}^{\infty}2kb_{k,n}e^{2ks}}{\sum_{k=1}^{\infty}b_{k,n}e^{2ks}},
    \end{equation}
    furthermore, 
    \begin{equation}
        g^{\prime\prime}(s)=\frac{1}{\rm log\,4}\frac{\sum_{k=1}^{\infty}4k^{2}b_{k,n}e^{2ks}\cdot\sum_{k=1}^{\infty}b_{k,n}e^{2ks}-(\sum_{k=1}^{\infty}2kb_{k,n}e^{2ks})^{2}}{(\sum_{k=1}^{\infty}b_{k,n}e^{2ks})^{2}}\geq 0
    \end{equation}
by Cauchy inequality. Hence $D(r)$ is monotone nondecreasing, and $D(r)\equiv k$ iff $w=a_{k}P_{k}$.
\end{proof}
Finally, we present a doubling index drop lemma for harmonic function, which is Lemma 2.20 in \cite{HJ1}. In roughly speaking, it tells us if the doubling index is far from some integer $m+1$, then it drops to the next integer $m$.
\begin{lemma}\label{sphere doubling drop}
    Let $h$ be a harmonic function in $B_{2}\subset \mathbb{R}^{n}$ and $m\in\mathbb{N}$, if  $D_{s}^{h}(0,1)\leq m+1-\epsilon$ with $\epsilon\leq c(n)$, then $D_{s}^{h}(0,\epsilon/2)\leq m+\epsilon$.
\end{lemma}

\section{Harmonic approximation and almost monotonicity for doubling index}
\subsection{Harmonic approximation lemma}
In this subsection, we prove the harmonic approximation for the solution to elliptic equation with continuous coefficients. First, we prove similar results to Lemma \ref{sphere doubling drop} for doubling index on ball of harmonic functions. 
\begin{lemma}\label{doubling index close to integer}
   Let $u$ be a harmonic function in $B_{10}(0)$ with $D(0,4)\leq \Lambda$, then for any $0<\epsilon\leq 1/1000$, there exists $\delta_{0}(\epsilon,n,\Lambda)$ such that if $\delta\leq \delta_{0}$ and $B_{2r}(x)\subset B_{1}(0)$ with $|D(x,2r)-D(x,r)|\leq \delta$, then\\
  \indent  $(1)$ there exists some integer $m\in\mathbb{N}$ such that $|D(x,r)-m|< \epsilon$, and \\
    \indent $(2)$ if $D(x,r)< m+1-\epsilon$ for some $m\in\mathbb{N}$, we have $D(x,2^{-1/\delta_{0}-1} r)< m+\epsilon$.
\end{lemma}
\begin{remark}
    For $(2)$, in the case of spherical doubling index, it is related to the frequency function $N(x,r)$, and a differential inequality implies the frequency drop, therefore the doubling index drop. However, when we deal with the doubling index on ball, we do not have such a direct differential inequality, but we can use the monotonicity and $(1)$ to obtain similar result.\end{remark}
\begin{proof}

    For $(1)$, we argue by contradiction. If not, there exist $\epsilon_{0}>0$ and $u^{i}$ with $|D^{u^{i}}(x_{i},2r_{i})-D^{u^{i}}(x_{i},r_{i})|\leq 1/i$, then denote the rescaled map $\hat{u}_{i}:=u_{x_{i},r_{i}}^{i}$, we have $|D^{\hat{u}_{i}}(0,2)-D^{\hat{u}_{i}}(0,1)|\leq \delta_{i}\to 0$ but $|D^{\hat{u}_{i}}(0,1)-m|\geq \epsilon_{0}$ for any integer $m$. Since $D^{u^{i}}(0,4)\leq \Lambda$, and $B_{2r_{i}}(x_{i})\subset B_{1}(0)$, then 
    \begin{equation}
        D^{\hat{u}_{i}}(0,2)=D^{u^{i}}(x_{i},2{r_{i}})\leq D^{u^{i}}(x_{i},2)
    \end{equation}
    by monotonicity of doubling index, since $x_{i}\in B_{1}(0)$, therefore
    \begin{equation}
        D^{u^{i}}(x_{i},2):= {\rm log}_{4}\frac{\fint_{B_{4}(x_{i})}|u^{i}|^{2}}{\fint_{B_{2}(x_{i})}|u^{i}|^{2}}\leq C+ {\rm log}_{4}\frac{\fint_{B_{8}(0)}|u^{i}|^{2}}{\fint_{B_{1}(0)}|u^{i}|^{2}}\leq C+3\Lambda.
    \end{equation}
Hence $D^{\hat{u}_{i}}(0,2)\leq C+3\Lambda$, since $\fint_{B_{1}(0)}|\hat{u}^{i}|^{2}=1$, then 
\begin{equation}
    \fint_{B_{2}(0)}|\hat{u}^{i}|^{2}\leq C(n,\Lambda).
\end{equation}
By interior estimate and Arzela-Ascoli theorem, we may assume $\hat{u}^{i}\to u_{\infty}$ smoothly up to a subsequence. Note that $u_{\infty}$ is harmonic with $D^{u_{\infty}}(0,2)=D^{u_{\infty}}(0,1)=N$, then by Lemma \ref{monotonicity for harmonic}, we have $u_{\infty}$ is a homogeneous harmonic polynomial with degree $N$, which contradicts $|D^{\hat{u}_{i}}(0,1)-m|\geq \epsilon_{0}$ for any integer $m$.\\
\indent For $(2)$, Since $D(x,r)\leq m+1/2$, then by $(1)$, 
\begin{equation}
    D(x,r)-D(x,r/2)>\delta_{0} \;\text{or}\; D(x,r/2)< m+\epsilon.
\end{equation}
For the second case, by monotonicity, $D(x,2^{-\frac{1}{\delta_{0}}-1}r)< m+\epsilon$. Hence, we consider the first case and, by $(1)$ again, we have 
\begin{equation}
    D(x,r/2)-D(x,r/4)>\delta_{0} \;\text{or}\; D(x,r/4)< m+\epsilon.
\end{equation}
Actually, we can repeat the process at most $M$ times, where $M=[\frac{1}{\delta_{0}}+1]$, then we conclude that $D(x,2^{-\frac{1}{\delta_{0}}-1}r)<m+\epsilon$ or $D(x,r)-D(x,2^{-\frac{1}{\delta_{0}}-1}r)>M\delta_{0}>1$, therefore  $D(x,2^{-\frac{1}{\delta_{0}}-1}r)<m+\epsilon$ or $D(x,2^{-\frac{1}{\delta_{0}}-1}r)D(x,r)-1<m-\epsilon$, and we have $D(x,2^{-\frac{1}{\delta_{0}}-1}r)<m+\epsilon$ in both cases.
\end{proof}
Next we introduce the harmonic approximation lemma, we can generalize the result in \cite{HJ1} for Hölder coefficients to continuous coefficients. 
\begin{lemma}\label{harmonic approximation}
     Let $u\in W^{1,2}(B_{4}(0))$ be a nonzero weak solution to \eqref{elliptic equation} with coefficient conditions \eqref{unifomly elliptic}, \eqref{continuous coefficient} and doubling condition \eqref{doubling assumption}. Let $\delta\in(0,1/100)$, there exists $r_{0}$ satisfying $\theta(r_{0})\leq C(n,\lambda,\theta)^{-\Lambda}\delta^{n+1}$ such that for any $x\in B_{1}(0)$ and $r\leq r_{0}$, there exists a harmonic function $h$ with $h(0)=0$, $\fint_{B_{1}(0)}h^{2}=1$ such that 
    \begin{equation}
        {\rm sup}_{B_{3}(0)}|u_{x,r}-h|\leq \delta.
    \end{equation}
\end{lemma}

\begin{proof}
Denote $\hat{u}=u_{x,r}$, then \( \int_{B_1(0)} \hat{u}^2 = 1 \). By doubling property, we have
\begin{equation}
\int_{B_8(0)} \hat{u}^2 \leq C(n, \lambda)^\Lambda.
\end{equation}
By elliptic interior estimate, we have
\begin{equation}
\sup_{B_6(0)} |\hat{u}|  \leq C(n, \lambda,\theta)^{\Lambda}.
\end{equation}
Let \( \hat{h} \) satisfy the following Dirichlet problem
\begin{equation}
\begin{cases}
\Delta {\hat{h}} = 0 & \text{in } B_4(0) \\
\hat{h} = \hat{u} & \text{on } \partial B_4(0).
\end{cases}
\end{equation}
Consider \( v= \hat{u} - \hat{h} \). Since \( \hat{u} \) satisfies the elliptic equation, then
\begin{equation}
-\Delta v = -\Delta \hat{u} = \partial_i ((\hat{a}^{ij} - \delta^{ij}) \partial_j \hat{u}):={\rm div}f.
\end{equation}
Where $\hat{a}(y)=a(x)^{-1/2}a(x+rA_{x}(y))a(x)^{-1/2}$ with $|\hat{a}(y)-\hat{a}(0)|\leq C\theta(\sqrt{1+\lambda}r|y|)$. Note that $v=0$ in $\partial B_{4}(0)$, then integration by parts and Young inequality imply
\begin{equation}
    \int_{B_{4}(0)}|\nabla v|^{2}\leq 2\int_{B_{4}(0)}f^{2}.
\end{equation}
Since 
\begin{equation}
    \int_{B_{4}(0)}f^{2}\leq [\theta(4\sqrt{(1+\lambda)})]^{2}\int_{B_{4}(0)}|\nabla\hat{u}|^{2},
\end{equation}
and by Caccioppoli inequality, $\int_{B_{4}(0)}|\nabla\hat{u}|^{2}\leq C\int_{B_{8}(0)}\hat{u}^{2}$, 
combining the above inequalities, we can conclude that
\begin{equation}\label{3.4.1}
\int_{B_4} v^2 \leq C^{\Lambda}[\theta(4\sqrt{1+\lambda}r)]^{2}.
\end{equation}
Recall ${\rm div}(\hat{A}\nabla\hat{u})=0$ and $A$ is continuous, then by $W^{1,p}$ estimate, we have
\begin{equation}
    ||\nabla \hat{u}||_{L^{p}(B_{4}(0))}\leq C(n,p,\lambda,\theta)||\hat{u}||_{L^{p}(B_{6}(0))}\leq C^{\Lambda},
\end{equation}
therefore $||f||_{L^{p}(B_{4}(0))}\leq C^{\Lambda}\theta(4\sqrt{(1+\lambda)})$. Let $p=2n$, since ${-\Delta v}={\rm div}f$, also by $W^{1,p}$ estimate and Sobolev embedding, we obtain
\begin{equation}
    ||v||_{C^{\alpha}(B_{7/2}(0))}\leq C||f||_{L^{2n}(B_{4}(0))}\leq C^{\Lambda},
\end{equation}
where $\alpha=1/2$. Denote \( E = \{z \in B_3 : |v(z)| \geq \epsilon \} \), where $\epsilon$ is small and to be determined. we claim that that $E=\emptyset$, therefore ${\rm sup}_{B_{3}}|v|\leq \epsilon$. If not, there exists some \( x_{0} \in E\). Since $|| v||_{C^{1/2}(B_{7/2}(0))}\leq C$, then for any $y\in B_{s}(x_{0})$ with $s^{1/2}= \epsilon/2C\leq 1/2$, we have 
\begin{align}
\begin{split}
|v(y)|&\geq |v(x_{0})|-|y-x_{0}|^{1/2}\cdot||v||_{C^{1/2}(B_{s}(x_{0}))}\\
&\geq \epsilon-Cs^{1/2}\geq \epsilon/2.
\end{split}
\end{align}
Therefore,
\begin{equation}\label{3.4.2}
\int_{B_{7/2}} v^2 \geq \int_{B_s(x_{0})} v^2 \geq C^{-\Lambda}\epsilon^{2n+2}.
\end{equation}
By \eqref{3.4.1}, if we choose $r\leq r_{0},\kappa=10(1+\lambda)$ with $\theta(\kappa r_{0})\leq C(n,\lambda,\theta)^{-\Lambda}\epsilon^{n+1}$, then $C^{\Lambda}[\theta(4\sqrt{1+\lambda}r)]^{2}< C^{-\Lambda}\epsilon^{2n+2}$, which contradicts \eqref{3.4.2}. Hence the set \( E = \emptyset \), then ${\rm sup}_{B_{3}(0)}||\hat{u}-\hat{h}||\leq\epsilon$.\\
\indent Since $\hat{u}(0)=0$, then $|\hat{h}(0)|\leq \epsilon$ and $\sqrt{\fint_{B_{1}}|\hat{h}|^{2}}\in[1-\epsilon,1+\epsilon]$, therefore $\sqrt{\fint_{B_{1}}|\hat{h}-\hat{h}(0)|^{2}}\in[1-2\epsilon,1+2\epsilon]$. Let
\begin{equation}
    h(y)=\frac{\hat{h}(y)-\hat{h}(0)}{\sqrt{\fint_{B_{1}}|\hat{h}(y)-\hat{h}(0)|^{2}}},
\end{equation}
then $\fint_{B_{1}}h^{2}=1$ and 
\begin{align}
\begin{split}
    {\rm sup}_{B_{3}}|\hat{u}-h|&\leq {\rm sup}_{B_{3}}|\hat{u}-(\hat{h}-\hat{h}(0))|+{\rm sup}_{B_{3}}|(\hat{h}-\hat{h}(0))-h|\\
    &\leq 2\epsilon+{\rm sup}_{B_{3}}|(\hat{h}-\hat{h}(0))-h|\leq 2\epsilon+{\rm sup}_{B_{3}}|(\hat{h}-\hat{h}(0))|\cdot 3\epsilon\\
    &\leq 2\epsilon+3\epsilon\cdot ({\rm sup}_{B_{3}}|(\hat{u}-\hat{u}(0))|+2\epsilon)\\
    &\leq 2\epsilon+3\epsilon(2C^{\Lambda}+1)\leq C\epsilon.
    \end{split}
\end{align}
Choosing $\epsilon=\delta/C^{\Lambda}$, then we finish the proof. And now we require $\theta(r_{0})\leq C^{-\Lambda}\delta^{n+1}$.
\end{proof}
Next, we will prove the approximation for doubling index.
\begin{lemma}\label{doubling index comparasion}
    Let $u\in W^{1,2}(B_{4}(0))$ be a nonzero weak solution to \eqref{elliptic equation} with coefficient conditions \eqref{unifomly elliptic}, \eqref{continuous coefficient} and doubling condition \eqref{doubling assumption}. Given $\epsilon\in (0,1/1000)$ and $\eta(n,\lambda,\Lambda,\theta,\epsilon)>0$, there exists $r_{0}=r(n,\lambda,\Lambda,\theta,\epsilon)$ such that for any $x\in B_{1}(0)$ and $r\leq r_{0}$, there exists a normalized harmonic function $h$ such that
    \begin{equation}
        |D^{u_{x,r}}(0,t)-D^{h}(0,t)|\leq \epsilon
    \end{equation}
for all $\eta\leq t\leq 1$. Especially, if $\eta=c(n)\epsilon^{C(n)}$, then we set $\kappa=10(1+\lambda)$ and choose $r_{0}$ such that $\theta(\kappa r_{0})\leq C^{-\Lambda}\epsilon^{C(n,\lambda)\Lambda}$.
\end{lemma}
\begin{proof}
    Let $\delta=\delta(\eta,\epsilon)$ be a small constant to be determined later. Then by Lemma \ref{harmonic approximation}, there exists $r_{0}$ such that for any $x\in B_{1}(0)$ and $r\leq r_{0}$, there exists a normalized harmonic function $h$ such that
    \begin{equation}
        {\rm sup}_{B_{4}(0)}|\hat{u}-h|\leq \delta.
    \end{equation}
    Where $\hat{u}:=u_{x,r}$, and by the monotonicity of $E(r)$, then for any $s\in[1,2]$, we have
    \begin{equation}
        \fint_{B_{s}}\hat{u}^{2}\geq \frac{1}{2}\fint_{B_{s}}h^{2}-{\rm sup}_{B_{2}}|\hat{u}-h|\geq \frac{1}{2}-\delta>1/4.
    \end{equation}
    Then for any $t\leq 1$, we can choose $i$ such that $1<2^{i+1}t\leq 2$, then by doubling condition,
    \begin{equation}
        \fint_{B_{t}}\hat{u}^{2}\geq 4^{-C(n,\lambda)\Lambda(i+1)}\fint_{B_{2^{i+1}t}}\hat{u}^{2}\geq ct^{2C\Lambda}\geq c\eta^{2C\Lambda}
    \end{equation}
    for $t\in [\eta,1]$. Now we take $\delta=c\epsilon^{3}\eta^{2C\Lambda}$, then $\fint _{B_{t}}|\hat{u}-h|^{2}\leq \epsilon^{3}\fint_{B_{t}}\hat{u}^{2}$, combined with 
    \begin{equation}
        \fint_{B_{t}}|\hat{u}+h|^{2}\leq 2(\fint_{B_{t}}|h-\hat{u}|^{2}+\fint_{B_{t}}|2\hat{u}|^{2})\leq 10\fint_{B_{t}}\hat{u}^{2},
    \end{equation}
we have
\begin{equation}
    \fint_{B_{t}}|\hat{u}^{2}-h^{2}|\leq \frac{1}{2}\fint_{B_{t}}(\frac{\epsilon}{50}|\hat{u}+h|^{2}+\frac{50}{\epsilon}|\hat{u}-h|^{2})\leq \frac{\epsilon}{2} \fint_{B_{t}}\hat{u}^{2}.
    \end{equation}
    Therefore,
    \begin{equation}
        |D^{\hat{u}}(0,t)-D^{h}(0,t)|\leq  {\rm log}_{4}\frac{1+\epsilon/2}{1-\epsilon/2}\leq \epsilon
    \end{equation}
    for any $\eta\leq t\leq 1$.
\end{proof}

\subsection{Almost monotonicity formula}
In this subsection, we prove the almost monotonicity for doubling index, which helps to obtain the volume estimates for quantitative singular strata. The proof follows the arguments of \cite{HJ1}.
\begin{theorem}\label{almost monotonicity for doubling index theorem}
       Let $u\in W^{1,2}(B_{4}(0))$ be a nonzero weak solution to \eqref{elliptic equation} with coefficient conditions \eqref{unifomly elliptic} and \eqref{continuous coefficient}. Assume that $u$ satisfies doubling assumption \eqref{doubling assumption}, then given $0<\epsilon\leq 1/1000$, there exists $r_{0}(n,\lambda,\Lambda,\theta,\epsilon)$ such that for any $x\in B_{1}(0)$ and $r\leq r_{0}$, we have
      \begin{equation}
          D(x,s)\leq D(x,r)+\epsilon
      \end{equation}
      for any $0<s\leq r$.
\end{theorem}
\begin{remark}
    If we fix $\epsilon<1/1000$, since $\lim_{s\to 0}D(x,s)=\mathcal{O}_{u}(x)\geq 1$, hence $D(x,r)\geq 1-\epsilon$ for any $r\leq r_{0}$.
\end{remark}
\begin{proof}[Proof of Theorem \ref{almost monotonicity for doubling index theorem}]
   Taking $\epsilon/100$ in Lemma \ref{doubling index close to integer}, then $\delta_{0}=\delta_{0}(n,\Lambda,\epsilon)$ is fixed, and we choose $\eta=2^{-\delta_{0}-1}$ in Lemma \ref{doubling index comparasion}. By Lemma \ref{doubling index comparasion}, for any $x\in B_{1}(0)$ and $r\leq r_{0}$, there exists a normalized harmonic function $h$ such that 
   \begin{equation}\label{3.7.1}
       |D^{u_{x,r}}(0,t)-D^{h}(0,t)|\leq \epsilon/100
   \end{equation}
   for all $\eta\leq t\leq 1$. Next we choose the integer $m$ such that $|D^{u}(x,r)-m|\leq 1/2$.\\
   \indent \textbf{We claim:} If $D(x,r)\leq m+\epsilon/10$, then $D(x,s)\leq m+3\epsilon/10$ for any $0< s\leq r$.
   \begin{proof}[proof of Claim]
       We define $\kappa_{0}$ as
       \begin{equation}
           \kappa_{0}:={\rm inf}\{t\geq 0: D(x,s)\leq m+\epsilon/5\;\text{for any }s\in [t,r].\}
       \end{equation}
       By \eqref{3.7.1}, we have 
       \begin{align}
       \begin{split}
           D(x,s)&=D^{u_{x,r}}(0,s/r)\leq D^{h}(0,s/r)\leq D^{h}(0,s/r)+\epsilon/100\\
           &\leq D^{h}(0,1)+\epsilon/100\leq D^{u_{x,r}}(0,1)+\epsilon/50\leq m+3\epsilon/10
           \end{split}
       \end{align}
for all $\eta r\leq s\leq r$. Therefore $\kappa_{0}\leq \eta r$. And our goal is to show $\kappa_{0}=0$. If $\kappa_{0}>0$, we apply Lemma \ref{doubling index comparasion} to $u_{x,\kappa_{0}}$, then there exists a normalized harmonic function $h_{0}$ such that 
\begin{equation}
    |D^{u_{x,\kappa_{0}}}(0,t)-D^{h_{0}}(0,t)|\leq \epsilon/100
\end{equation}
for all $\eta\leq t\leq 1$. Similarly, by the monotonicity of $D^{h_{0}}$, we have 
\begin{align}
    \begin{split}
        D(x,t\kappa_{0})&\leq D^{h_{0}}(0,t)+\epsilon/100\leq D^{h_{0}}(0,1)+\epsilon/100\\
        &\leq D^{u}(x,\kappa_{0})+\epsilon/50\leq  m+3\epsilon/10.
    \end{split}
\end{align}
for all $\eta\leq t\leq 1$. Note that $D^{h_{0}}(0,1)\leq D(x,\kappa_{0})+\epsilon/100\leq m+\epsilon/3<m+1-\epsilon/100$, then by Lemma \ref{doubling index close to integer}, $D^{h_{0}}(0,\eta)\leq m+\epsilon/100$. Therefore $D(x,\eta\kappa_{0})\leq m+\epsilon/50\leq m+\epsilon/10$. Then we define $\kappa_{1}$ as 
\begin{equation}
     \kappa_{1}:={\rm inf}\{t\geq 0: D(x,s)\leq m+\epsilon/5\;\text{for any }s\in [t,\eta\kappa_{0}]\}.
\end{equation}
Repeating  the above arguments, we obtain $D(x,s)\leq m+3\epsilon/10$ for any $s\in[\eta\kappa_{1},r]$, note that $\eta\kappa_{1}\leq \eta^{2}\kappa_{0}$, and we can continue this process to obtain $\kappa_{j}$ such that 
\begin{equation}
    D(x,s)\leq m+3\epsilon/10\;\text{for any}\;s\in[\eta\kappa_{j},r].
\end{equation}
Since $\eta\kappa_{j}\leq \eta^{j+1}\kappa_{0}\to 0$, then we conclude that for any $s\in(0,r]$, we have $D(x,s)\leq m+3\epsilon/10$. Therefore, the claim holds.\end{proof}
   Now we consider the following three cases.\\
   \indent$(a)$ If $m+\epsilon/10\leq D(x,r)\leq m+1/2$, then by Lemma \ref{doubling index comparasion}, for any $t\in[\eta,1]$, we have $|D^{u_{x,r}}(0,t)-D^{h}(0,t)|\leq \epsilon/100$ for some harmonic function $h$, therefore for any $s\in[\eta r,r]$,
   \begin{equation}
       D(x,s)=D^{u_{x,r}}(0,s/r)\leq D^{h}(0,s/r)+\epsilon/100\leq D^{h}(0,1)+\epsilon/100\leq D^{u}(x,r)+\epsilon/50.
   \end{equation}
   For $s\in(0,\eta r]$, note that $D^{h}(0,1)\leq D(x,r)+\epsilon/100< m+1-\epsilon/100$, by Lemma \ref{doubling index close to integer}, $D^h(0,\eta )< m+\epsilon/100$, therefore $D(x,\eta r)\leq D^{h}(0,\eta)+\epsilon/100\leq m+\epsilon/50$. Now we apply the claim to $D(x,\eta r)$ and conclude that 
\begin{equation}
    D(x,s)\leq m+3\epsilon/10\leq D(x,r)+\epsilon.
\end{equation}
for any $0<s\leq\eta r$.\\
\indent $(b)$ If $m-\epsilon/50\leq D(x,r)\leq m+\epsilon/10$, applying the claim to $D(x,r)$, then 
\begin{equation}
    D(x,s)\leq m+3\epsilon/10\leq D(x,r)+\epsilon.
\end{equation}
for any $s\in(0,r]$.\\
\indent $(c)$ If $m-1/2\leq D(x,r)< m-\epsilon/50$, by Lemma \ref{doubling index comparasion}, for any $t\in[\eta,1]$, we have $|D^{u_{x,r}}(0,t)-D^{h}(0,t)|\leq \epsilon/100$ for some harmonic function $h$, similar to $(a)$, we have $D(x,s)\leq D^{u}(x,r)+\epsilon/50$ for any $s\in[\eta r,r]$. And $D^{h}(0,1)\leq D^{u}(x,r)+\epsilon/100<m-\epsilon/100$, then by Lemma \ref{doubling index close to integer}, $D^{h}(0,\eta)< m-1+\epsilon/100$, therefore $D(x,\eta r)\leq D^{h}(0,\eta)+\epsilon/100\leq m-1+\epsilon/50$, then applying the claim to $D(x,\eta r)$, we have for any $s\in(0,\eta r]$, $D(x,s)\leq m-1+3\epsilon/10\leq D(x,r)+\epsilon$.
\end{proof}
As a direct corollary of the almost monotonicity theorem, we can show that the number of non-pinched scales is finite.
\begin{corollary}\label{finite nonpinched corollary}
    Under the same setting as in Theorem \ref{almost monotonicity for doubling index theorem}. Given $\epsilon\in(0,1/1000)$, $x\in B_{1}(0)$ and choose $r_{0}$ as in Theorem \ref{almost monotonicity for doubling index theorem} for $\epsilon/2$. Denote $T_{\epsilon}:=\{i\in\mathbb{N}:|D(x,\frac{r_{0}}{2^{i}})-D(x,\frac{r_{0}}{2^{i+1}})|>\epsilon\}$, then $|T_{\epsilon}|\leq C(n,\lambda)\Lambda/\epsilon$.
\end{corollary}
\begin{proof}
    First, by doubling assumption, $D(x,1)\leq C(n,\lambda)\Lambda$. We may assume $T_{\epsilon}=\{0\leq l_{1}<l_{2}<...\}$, given $\epsilon/2>0$, by the choice of $r_{0}$, we have
    \begin{equation}\label{3.9.1}
        \sum_{i=0}^{\infty}(D(x,\frac{r_{0}}{2^{l_{i}}})-D(x,\frac{r_{0}}{2^{l_{i+1}}}))\leq D(x,r_{0})-D(x,0)+\epsilon/2\leq C(n,\lambda)\Lambda.
    \end{equation}
    Note that for any $l_{i}\in T_{\epsilon}$, applying Theorem \ref{almost monotonicity for doubling index theorem} with $\epsilon/2$, then
    \begin{equation}
        D(x,\frac{r_{0}}{2^{l_{i}}})>D(x,\frac{r_{0}}{2^{l_{i}+1}})+\epsilon\geq D(x,\frac{r_{0}}{2^{l_{i+1}}})+\epsilon-\epsilon/2\geq D(x,\frac{r_{0}}{2^{l_{i+1}}})+\epsilon/2.
    \end{equation}
    Therefore by \eqref{3.9.1}, $|T_{\epsilon}|\cdot\epsilon/2\leq C(n,\lambda)\Lambda$, hence we conclude that $|T_{\epsilon}|\leq 2C(n,\lambda)\Lambda/\epsilon$.    
\end{proof}
\indent At the end of this section, we point out that following similar arguments with minor modifications, all properties of doubling index $D_{b}(x,r)$ in this section can be similarly proved for $D_{s}(x,r)$. Replace Lemma \ref{doubling index close to integer} with Lemma \ref{sphere doubling drop}, we can also obtain an almost monotonicity formula for doubling index on sphere, with an explicit dependence of $r_{0}$ regarding $\Lambda$.
\begin{theorem}\label{sphere almost monotonicity}
    Let $u\in W^{1,2}(B_{4}(0))$ be a nonzero weak solution to \eqref{elliptic equation} with coefficient conditions \eqref{unifomly elliptic} and \eqref{continuous coefficient}. Assume that $u$ satisfies doubling assumption \eqref{doubling assumption}, then given $0<\epsilon\leq 1/1000,\kappa=10(1+\lambda)$, there exists $r_{0}$ satisfying $\theta(\kappa r_{0})\leq C(n,\lambda,\theta)^{-\Lambda}\epsilon^{C(n,\lambda)\Lambda}$ such that for any $x\in B_{1}(0)$ and $0<r_{2}\leq r_{1}\leq r_{0}$, we have
      \begin{equation}
          D_{s}(x,r_{2})\leq D_{s}(x,r_{1})+\epsilon.
      \end{equation}
\end{theorem}
As a corollary, we prove a similar doubling index drop lemma to Lemma \ref{doubling index close to integer} for general elliptic equations, which plays an important role in the volume estimates for critical sets in Section 7. 
\begin{corollary}\label{cor3.9}
    Under the same setting in Theorem \ref{sphere almost monotonicity}, for any $x\in B_{1}$, if $D_{s}(x,r_{1})\leq m-\epsilon_{1}$ with $r_{1}\leq r_{0}$ and $\theta(\kappa r_{0})\leq (c\epsilon_{1})^{C\Lambda}$, then for any $r\leq \epsilon_{1} r_{1}/4$, we have $D_{s}(x,r)\leq m-1+\epsilon_{1}$.
\end{corollary}
\begin{proof}
    Let $w=\tilde{u}_{x,r_{1}}$, then taking $\eta=2\epsilon_{1}/5, \epsilon=\epsilon_{1}/10$ in Lemma \ref{doubling index comparasion}, by Lemma \ref{harmonic approximation} and Lemma \ref{doubling index comparasion} of sphere version, assume that $r_{0}$ satisfies $\theta(\kappa r_{0})\leq (c_{1}\epsilon_{1})^{C_{1}\Lambda}$, then if $r_{1}\leq r_{0}$, there exists a normalized harmonic function $h$ such that
    \begin{equation}\label{3.40}
        |D_{s}^{w}(0,t)-D_{s}^{h}(0,t)|\leq \epsilon_{1}/10
    \end{equation}
    for any $t\in[2\epsilon_{1}/5,1]$. Hence $D_{s}^{h}(0,1)\leq m-4\epsilon_{1}/5$, by Lemma \ref{sphere doubling drop}, we have $D_{s}^{h}(0,2\epsilon_{1}/5)\leq m-1+4\epsilon_{1}/5$. By \eqref{3.40}, $D_{x}^{w}(0,2\epsilon_{1}/5)=D_{s}^{u}(x,2\epsilon_{1}r_{1}/5)\leq m-1+9\epsilon_{1}/10$. Now taking $\epsilon=\epsilon_{1}/10$ in Theorem \ref{almost monotonicity for doubling index theorem}, if $\theta(\kappa r_{0})\leq (c_{2}\epsilon_{1})^{C_{2}\Lambda}$, since $r_{1}\leq r_{0}$, then for any $r\leq \epsilon_{1}r_{1}/4$, we have $D_{s}^{u}(x,r)\leq D_{s}^{u}(x,2\epsilon_{1}r_{1}/5)+\epsilon_{1}/10\leq m-1+\epsilon_{1}$.
\end{proof}

\section{Quantitative stratification and uniform volume estimates for geometric sets}

\subsection{Quantitative stratification and cone splitting principle}
In this subsection, we introduce the concepts of quantitative symmetry and stratification.
\begin{definition}
    Let $u:\mathbb{R}^{n}\to \mathbb{R}$ be a continuous function.\\
    \indent $(a)$ We say $u$ is $0$-symmetric at $x_{0}$ if there exists $d\in\mathbb{N}$ such that $u(x_{0}+\lambda x)=\lambda^{d}u(x_{0}+x)$ for any $x\in\mathbb{R}^{n}$ and $\lambda>0$, i.e., $u$ is a homogeneous polynomial at $x_{0}$.\\
    \indent $(b)$ We say $u$ is $k$-symmetric at $x_{0}$ if $u$ is $0$-symmetric at $x_{0}$ and there exists a subspace $V$ of dimension $k$ such that
    $u(x_{0}+x+y)=u(x_{0}+x)$ for any $x\in\mathbb{R}^{n}$ and $y\in V$.\\
   \indent $(c)$ Denote $\hat{u}_{x,r}(y)=\frac{u(x+ry)-u(x)}{\sqrt{\fint_{B_{1}(0)}|u(x+ry)-u(x)|^{2}}}$, and we say $u$ is $(k,\epsilon)$-symmetric on $B_{r}(x)$ if there exists some $k$-symmetric $P$ such that 
    \begin{equation}
        ||\hat{u}_{x,r}-P||_{L^{\infty}(B_{1/{\sqrt{1+\lambda}}}(0))}\leq \epsilon, 
    \end{equation}where $P$ satisfies some uniform elliptic equation with constant coefficients and $\fint_{B_{1}(0)}|P|^{2}\geq c(n,\lambda,\Lambda)>0$.\\
    \indent $(d)$ We define the quantitative singular stratum by 
    \begin{equation}
        S_{\eta,r}^{k}(u):=\{x\in B_{1}(0):u \;\text{is not $(k+1,\eta)$-symmetric on $B_{s}(x)$  for any $r\leq s\leq 1$}.\}
    \end{equation}
\end{definition}
\begin{remark}
    For $(c)$, it actually follows $||u_{x,r}-Q||_{L^{\infty}(B_{1}(0))}\leq C\epsilon$, where $Q$ is a normalized hhp, i.e., $Q$ is a homogeneous harmonic polynomial  with $\fint_{B_{1}(0)}Q^{2}=1$, and then $P(z)=b(r)Q(A_{x}^{-1}(z))$, where $0<c_{0}\leq b(r)\leq C_{0}$. Although this form seems a little complicated, it is more convenient for us to state the cone splitting principle.
\end{remark}
First, we prove that the doubling index $\delta$-pinched at $x$ implies the almost $0$-symmetric property.
\begin{theorem}\label{0,epsilon symmetric thm}
     Let $u\in W^{1,2}(B_{4}(0))$ be a nonzero weak solution to \eqref{elliptic equation} with coefficient conditions \eqref{unifomly elliptic} and \eqref{continuous coefficient}. Assume that $u$ satisfies doubling property \eqref{doubling assumption}, then given $\beta\in(0,1)$ and $\delta>0$, there exist $\epsilon_{0},r_{0}$, depending only on $n,\lambda,\theta,\Lambda,\delta,\beta$ such that if $|D^{u}(x,r)-D^{u}(x,\beta r)|\leq \epsilon_{0}$ with $x\in B_{1}(0)$ and $r\leq r_{0}$, then there exists a normalized hhp $Q$ such that
    \begin{equation}
        ||u_{x,r}-Q||_{L^{\infty}(B_{1}(0))}\leq \delta.
    \end{equation}
 And therefore $u$ is $(0,C(n,\lambda,\Lambda)\delta)$-symmetric on $B_{r}(x)$.
\end{theorem}
\begin{proof}
    We argue by contradiction. If not, there exists $\delta_{0}>0$
such that for $\epsilon_{k}=1/k$ and $r_{0,k}\leq 1/k$, there exist a sequence $\{u^{k}\}$ satisfying elliptic equations, points $\{x_{k}\}\subset B_{1}(0)$ and positive numbers $r_{k}\leq 1/k$ such that 
\begin{equation}
    |D^{v_{k}}(0,1)-D^{v_{k}}(0,\beta )|\leq 1/k,
\end{equation}
but for any normalized hhp $Q$,
\begin{equation}\label{4.3.1}
    ||v_{k}-Q||_{L^{\infty}(B_{1}(0))}>\delta_{0}.
\end{equation}
Where $v_{k}=u_{x_{k},r_{k}}^{k}$. Since $\fint_{B_{1}(0)}|v_{k}|^{2}=1$, then by doubling condition,
\begin{equation}
    \fint_{B_{2}(0)}|v_{k}|^{2}\leq C(n,\lambda,\Lambda).
\end{equation}
By gradient estimate in Theorem \ref{DMO regularity}, and combing the Caccioppoli inequality, we have 
\begin{equation}
    ||v_{k}||_{C^{1}(B_{3/2}(0))}\leq C(n,\lambda,\theta,\Lambda).
\end{equation}
\indent Therefore, up to a subsequence, we may assume $v_{k}\to v$ in $C^{0}(B_{1}(0))\cap W^{1,2}(B_{1}(0))$ with $\fint_{B_{1}(0)}v^{2}=1$. Recall that $v_{k}$ satisfies ${\rm div}(\hat{A}_{k}\nabla v_{k})=0$, where $\hat{A}_{k}=a^{k}(x_{k}+r_{k}A_{x_{k}}^{k}(y))\cdot a^{k}(x_{k})^{-1}$, and $A_{x}^{k}=\sqrt{a^{k}(x)}^{-1}$, since $a^{k}$ is uniformly elliptic and with uniform modulus of continuity $\theta$, recall $r_{k}\to 0$, then we have $\hat{A}^{k}\to Id$ as $k\to \infty$. Therefore, $v$ is a harmonic function, and since $D^{v}(0,1)=D^{v}(0,\beta)$, by Lemma \ref{monotonicity for harmonic}, $v$ is some hhp $Q$ in $B_{1}(0)\setminus B_{\beta}(0)$, and by unique continuation, then $v\equiv Q$, which contradicts to \eqref{4.3.1}. Therefore we have
\begin{equation}
    ||u_{x,r}-Q||_{L^{\infty}(B_{1}(0))}\leq \delta,
\end{equation}
i.e.,
\begin{equation}
    ||\frac{u(x+rA_{x}(y))-u(x)}{\sqrt{\fint_{B_{1}(0)}|u(x+rA_{x}(y))-u(x)|^{2}}}-Q(y)||_{L^{\infty}(B_{1}(0))}\leq \delta.
\end{equation}
\indent Let $w=A_{x}(y)$, $P(w)=b(r)Q(A_{x}^{-1}(w))$, where $b(r)^{2}=\fint_{B_{1}(0)}|u(x+rA_{x}(y))-u(x)|^{2}dy/\fint_{B_{1}(0)}|u(x+rw)-u(x)|^{2}dw$, by doubling property and the inclusion relationship, we have $c_{0}\leq b(r)\leq C_{0}$, hence we conclude that
\begin{equation}
    ||\hat{u}_{x,r}-P||_{L^{\infty}(B_{1/\sqrt{1+\lambda}}(0))}\leq C_{0}(n,\lambda,\Lambda)\delta.
\end{equation}
\end{proof}
In the following, we show a qualitative cone splitting principle.
\begin{lemma}\label{qualitative cone splitting lem}
    Let $u$ be a nonzero continuous function, if $u$ is $k$-symmetric at $0$ with respect to a subspace  $V$ with $k\leq n-2$, and $u$ is $0$-symmetric at $y\notin V$, then $u$ is $(k+1)$-symmetric at $0$ with respect to $W={\rm span}(V,y)$.
\end{lemma}
\begin{remark}
    If $u$ satisfies elliptic equation ${\rm div}(A\nabla u)=0$ with $A$ is a constant matrix, and assume that $u$ is $(n-1)$-symmetric, then actually $u$ essentially has only one variable, therefore $u$ is a linear function.
\end{remark}
\begin{proof}
    \indent We may assume $u=P_{d}(x)$ with $P$ is $k$-symmetric by definition, since $u$ is $0$-symmetric at $y$, then $u(y+\lambda x)=\lambda^{m}u(y+x)$ for some integer $m$, on the other hand,
    \begin{equation}
        u(y+\lambda x)=P_{d}(y+\lambda x)=\lambda^{d}P_{d}(x+y/\lambda).
    \end{equation}
    \indent We claim that $d=m$, if not we may assume $m>d$, then $P_{d}(x+y)=\lambda^{d-m}P_{d}(x+y/\lambda)$ holds for any $x\in \mathbb{R}^{n}$ and $\lambda>0$. Now let $\lambda\to\infty$, we have $P_{d}\equiv 0$, which is a contradiction. Therefore we conclude that
    \begin{equation}
        u(x+y)=P_{d}(x+y)=P_{d}(x+y/\lambda),
    \end{equation}
    and let $\lambda\to 0$, it implies $P_{d}(x)=P_{d}(x+y)$, combined with above equality, we have $P_{d}(x)=P_{d}(x+\mu y)$ for any $\mu>0$ and $x\in \mathbb{R}^{n}$, therefore $u(x)=P_{d}(x)$ is $(k+1)$-symmetric with respect to $W={\rm span}(V,y)$.
\end{proof}
Now we are ready to prove the effective cone splitting principle.
\begin{proposition}\label{cone splitting prop}
    Under the same assumptions in Theorem \ref{0,epsilon symmetric thm}, given $\tau,\eta>0$, then there exists $\delta_{0}$ and $r_{0}$, depending only on $n,\lambda,\theta,\Lambda,\tau,\eta$ such that for any $x\in B_{1}(0)$, $r\leq r_{0}$ and $\delta\leq \delta_{0}$ if \\
   \indent $(1)$ $u$ is $(k,\delta)$-symmetric on $B_{r}(x)$ with respect to a $k$-plane $V_{x}:=x+V$ with $k\leq n-2$ and\\
   \indent $(2)$ $u$ is $(0,\delta)$-symmetric on $B_{r}(y)$ with $y\in B_{r}(x)\setminus B_{\tau r}(V_{x})$, then we have $u$ is $(k+1,\eta)$-symmetric on $B_{r}(x)$.
\end{proposition}
\begin{proof}
    We also argue by contradiction. If not, there exists $\eta_{0}>0$ such that for $\delta_{0,i}=r_{0,i}=1/i$, there exist $u^{i}$, $x_{i}\in B_{1}(0)$ and $r_{i}\leq 1/i$ such that $||\hat{u}_{x_{i},r_{i}}^{i}-\hat{P}_{i}||_{L^{\infty}(B_{1/\sqrt{1+\lambda}})}\leq \delta_{i}$, i.e.,
    \begin{equation}
        ||\frac{u^{i}(x_{i}+r_{i}w)-u^{i}(x_{i})}{\sqrt{\fint_{B_{1}(0)}|u(x_{i}+r_{i}w)-u^{i}(x_{i})|^{2}dw}}-\hat{P}_{i}(w))||_{L^{\infty}(B_{1/\sqrt{1+\lambda}})}\leq \delta_{i}\leq 1/i,
    \end{equation} 
where $\hat{P}_{i}$ is $k$-symmetric, but $u^{i}$ is not $(k+1,\eta_{0})$-symmetric on $B_{r_{i}}(x_{i})$. Based on Theorem \ref{0,epsilon symmetric thm}, we may assume that $\hat{P}_{i}=P_{i}(A_{i}^{-1})$, where $\{P_{i}\}_{i\geq 1}$ is a sequence of normalized hhps and $\{A_{i}\}_{i\geq 1}$ is a sequence of constant matrices with the same uniform ellipticity condition. Using similar arguments as in Theorem \ref{0,epsilon symmetric thm}, by doubling condition and interior estimate, we may assume that
\begin{equation}\label{hat u convergence}
        \hat{u}_{x_{i},r_{i}}^{i}\to u_{\infty}\;\text{in}\; C^{0}(B_{1}(0))\;\text{as}\;i\to \infty.
    \end{equation}
 Besides, let $P_{i}\to P_{\infty}, A_{i}\to A_{\infty}$, and we denote $P(w)=P_{\infty}(A_{\infty}^{-1}(w))$ satisfying the second order elliptic equation of constant coefficients. Therefore, $u_{\infty}=P$. By similar arguments, we also have
    \begin{equation}
        \hat{u}_{y_{i},r_{i}}^{i}\to v_{\infty} \;\text{in}\;C^{0}(B_{1}(0))\;\text{as}\;i\to \infty
    \end{equation}
    and $\hat{Q}_{i}(w)\to Q(w)$, where $||\hat{u}_{y_{i},r_{i}}-\hat{Q}_{i}||\leq \delta_{i}\to 0$, so $v_{\infty}=Q$. And note that 
    \begin{equation}\label{transformation}
        a_{i}\hat{u}_{x_{i},r_{i}}^{i}(w+\frac{y_{i}-x_{i}}{r_{i}})=\hat{u}_{y_{i},r_{i}}^{i}(w)\to v_{\infty}(w),
    \end{equation}
    where by doubling property and $|x_{i}-y_{i}|\leq r_{i}$
    \begin{equation}
         a_{i}=\frac{\sqrt{\fint_{B_{r_{i}}(x_{i})}|u^{i}|^{2}}}{\sqrt{\fint_{B_{r_{i}}(y_{i})}|u^{i}|^{2}}}\leq  c(n)\frac{\sqrt{\fint_{B_{2r_{i}}(y_{i})}|u^{i}|^{2}}}{\sqrt{\fint_{B_{r_{i}}(y_{i})}|u^{i}|^{2}}}\leq C(n,\Lambda).
    \end{equation}
    Similarly $1/C(n,\Lambda)\leq a_{i}$, therefore, we may assume $a_{i}\to a$ up to a subsequence. Recall that $y_{i}\in  B_{r_{i}}(x_{i})\setminus B_{\tau r_{i}}(x_{i}+V_{i})$, where $V_{i}$ is the invariant subspace of $\hat{P}_{i}$. Let $V_{i}\to V$, and $z_{i}:=\frac{y_{i}-x_{i}}{r_{i}}\to z$, then $z\in B_{1}(0)\setminus B_{\tau}(V)$. By \eqref{hat u convergence} and  \eqref{transformation}, then let $i\to\infty$, we have
    \begin{equation}
        au_{\infty}(w+z)=v_{\infty}(w).
    \end{equation}
    Recall that $v_{\infty}=Q,u_{\infty}=P$, then $aP(w+z)=Q(w)$ and $P$ is invariant with respect to the $k$-dimensional subspace $V$, $Q$ is $0$-symmetric, then by Lemma \ref{qualitative cone splitting lem}, $P$ is $(k+1)$-symmetric with respect to $W=span(V,z)$. Note that $P$ satisfies some elliptic equation of constant coefficients, which contradicts the assumption that $u^{i}$ is not $(k+1,\eta_{0})$-symmetric on $B_{r_{i}}(x_{i})$.
\end{proof}
Given $\beta>0$, we define $\epsilon$-pinched set at scale $r$ as
\begin{equation}
    \Sigma:=\{x\in B_{1}(0):|D^{u}(x,r)-D^{u}(x,\beta r)|\leq \epsilon\}.
\end{equation}
And we say a subset $S\subset B_{r}(x)$ is $(k,\tau)$-independent in $B_{r}(x)$ if for any affine plane $L$ of dimension $\leq k-1$, there exists some $y\in S$ such that $d(y,L)>\tau r$. As a direct corollary of Proposition \ref{cone splitting prop}, by an induction argument, we have the following result.
\begin{corollary}\label{quantitative cone splitting corollary}
    Under the same assumptions as in Theorem \ref{0,epsilon symmetric thm}, given $\eta,\tau>0$, there exist $\epsilon_{0},r_{0}$, depending only on $n,\lambda,\theta,\Lambda,\eta,\beta,\tau$ such that if  $r\leq r_{0}$ and the $\epsilon$-pinched set $\Sigma$ is $(k,\tau)$-independent in $B_{r}(x)$ with some $x\in\Sigma$ and $\epsilon\leq \epsilon_{0}$, then $u$ is $(k,\eta)$-symmetric on $B_{r}(x)$.
\end{corollary}
\subsection{Weak estimates for singular strata}
In this subsection, based on Theorem \ref{almost monotonicity for doubling index theorem} and Corollary \ref{quantitative cone splitting corollary}, i.e., the almost monotonicity theorem and quantitative cone splitting principle, we can follow the stratification and inductive covering arguments in \cite{CNV} to show a Minkowski-type volume estimate for nodal set. First we prove the volume estimates for quantitative singular strata, note that here we only need to assume $A$ is continuous.
\begin{theorem}\label{singular stratum volume estimate}
      Let $u\in W^{1,2}(B_{4}(0))$ be a nonzero weak solution to \eqref{elliptic equation}, with the coefficients satisfying \eqref{unifomly elliptic} and \eqref{continuous coefficient} in $B_{4}(0)$. Assume that $u$ satisfies the doubling condition \eqref{doubling assumption}, then there exists $r_{0}$ such that for any $\delta,\eta>0$, $0\leq k\leq n-2$ and $0<r\leq 1$, we have
     \begin{equation}
           {\rm Vol}(B_{r}(S_{\eta,r}^{k})\cap B_{1}(0))\leq Cr^{n-k-\delta},
     \end{equation}
     where $C$ depends only on $n,\lambda,\theta,\Lambda,\delta,\eta$.
\end{theorem}
\begin{remark}
    In Sections 5 and 6, if $A$ is Dini continuous, we can prove the quantitative uniqueness of tangent map and quantitative cone splitting. Based on these ingredients, one can obtain a sharp volume estimate for singular stratum, i.e., removing $\delta$ in the above inequality, by the same neck decomposition arguments as in \cite{HJ1}.
\end{remark}
\begin{proof}
    Given $r>0$, we consider the covering of $S_{\eta,r}^{k}$ with the balls of radius $r/5$. And let 
    \begin{equation}
        N_{r}:={\rm inf} \{K:S_{\eta,r}^{k}\subset\cup_{i=1}^{M}B_{r/5}(x_{i}),x_{i}\in B_{1}(0)\}.
    \end{equation}
    Then $B_{r}(S_{\eta,r}^{k})\subset \cup_{i=1}^{N_{r}}B_{2r}(x_{i})$, to obtain the volume estimate, it suffices to give an upper bound of $N_{r}$. Given $\eta,\delta>0$, then we fix $\beta=\gamma$ in Theorem \ref{0,epsilon symmetric thm} and $\tau=\gamma/10$ in Corollary \ref{quantitative cone splitting corollary}. Note that $S_{\eta,r_{1}}^{k}\subset S_{\eta,r_{2}}^{k}$ if $r_{1}\leq r_{2}$, therefore it suffices to consider the discrete cases $r=r_{0}\gamma^{j}$ for all $j\geq 0$, where $\gamma=\gamma(n,\delta)$ is a universal constant to be determined later and $r_{0}=r_{0}(n,\lambda,\theta,\Lambda,\eta,\gamma)$ is taken as in Corollary \ref{quantitative cone splitting corollary}.\\
    \indent Consider $r=\gamma^{l}$, for any $x\in B_{1}(0)$, we define $T_{\epsilon}^{j}(x)\in\mathbb{R}^{j}$, where each term in $T_{\epsilon}^{j}(x)$ is chosen from $\{0,1\}$. The $p$-th term $a_{p} (p\geq 1)$ in $T_{\epsilon}^{j}(x)$ is defined as $0$ if 
    \begin{equation}
        |D(x,r_{0}\gamma^{p})-D(x,r_{0}\gamma^{p+1})|\leq \epsilon,
    \end{equation}
    and as $1$ if 
    \begin{equation}
        |D(x,r_{0}\gamma^{p})-D(x,r_{0}\gamma^{p+1})|> \epsilon.
    \end{equation}
    \indent For $i\leq j$, $T_{\epsilon}^{i}(x)$ can be naturally viewed as a restriction of $T_{\epsilon}^{j}(x)$ to $\mathbb{R}^{i}$. And note that by Corollary \ref{finite nonpinched corollary}, we have the number of non-pinched scales is uniformly bounded, i.e., 
    \begin{equation}
        |T_{\epsilon}^{j}(x)|:=\#\{a_{p}=1,\text{for}\; 1\leq p\leq j\}\leq M(n,\lambda,\theta,\Lambda,\eta,\gamma)
    \end{equation}
    for any $j\geq 1$. Denote $\Omega_{j}$ to be all $T\in \{0,1\}^{j}$ with $|T|\leq M$, and for any $T^{j}\in\Omega_{j}$, we define
    \begin{equation}
        E_{T^{j}}:=\{x\in B_{1}(0):T_{\epsilon}^{j}(x)=T^{j}\}.
    \end{equation}
    Then we can rewrite $B_{1}(0)$ as a union in the following way
    \begin{equation}
        \cup_{T^{j}\in \Omega_{j}}E_{T^{j}}=B_{1}(0).
    \end{equation}
    Let $|\Omega_{j}|$ denote the number of elements in $\Omega_{j}$, then it is bounded by $C_{j}^{M}<j^{M}$. \\
    \indent Next given some $T\in \Omega_{l}$, we focus on estimating the volume of $E_{T}\cap S_{\eta,r}^{k}$, where $r=r_{0}\gamma^{l}$, and $j+1\leq l$. We will use the balls centered at $E_{T}\cap S_{\eta,r}^{k}$ and radius $r_{0}\gamma^{j}$ to cover it, and now let 
    \begin{equation}
         N_{j,T}:={\rm inf} \{K_{j}:E_{T}\cap S_{\eta,r}^{k}\subset\cup_{i=1}^{K_{j}}B_{r_{0}\gamma^{j}}(x_{a}^{j})\;\text{with}\;x_{a}^{j}\in E_{T}\cap S_{\eta,r}^{k}\}.
    \end{equation}
Then we have 
\begin{align}\label{4.8.1}
\begin{split}
      {\rm Vol}(B_{r}(S_{\eta,r}^{k})\cap B_{1}(0))&\leq \sum_{T^{l}\in \Omega_{l}}  {\rm Vol}(B_{r}(S_{\eta,r}^{k}\cap E_{T^{l}})\cap B_{1}(0))\leq \sum_{T^{l}\in \Omega_{l}}  {\rm Vol}(B_{r}(S_{\eta,r}^{k}\cap E_{T^{l}}))\\
    &\leq \sum_{T^{l}\in\Omega_{l}}N_{l,T^{l}}\cdot\omega_{n}(2r_{0}\gamma^{l})^{n}\leq \omega_{n}N_{l}l^{M}(2r_{0}\gamma^{l})^{n},
    \end{split}
\end{align}
where $N_{j}={\rm sup}_{T^{l}\in {\Omega_{l}}}N_{j,T^{l}}$. \\
\indent Now it suffices for us to give an upper bound of $N_{l,T_{l}}$. For our convenient, denote $T=T^{l}$ and we will estimate $N_{j+1,T}$ by $N_{j,T}$. First for each $\Sigma_{j}:=B_{r_{0}\gamma^{j}}(x_{a}^{j})\cap E_{T}\cap S_{\eta,r}^{k}$, we will choose the balls $\{B_{r_{0}\gamma^{j+1}}(x_{a}^{j+1})\}$ with $x_{a}^{j+1}\in \Sigma_{j}$ to cover $\Sigma_{j}$. Let $T=(a_{1},a_{2},...,a_{l})$ and we consider the following two cases.\\
\indent \textbf{Case 1:} if $a_{j}=1$, we only to choose $\{B_{r_{0}\gamma^{j+1}}(x_{a}^{j+1})\}$ to cover $B_{r_{0}\gamma^{j}}(x_{a}^{j})$, then a direct covering implies $N_{j+1,T}\leq c(n)\gamma^{-n}N_{j,T}$.\\
\indent \textbf{Case 2}: if $a_{j}=0$, for any $x\in \Sigma_{j}$, by definition, we have 
\begin{equation}
    |D(x,r_{0}\gamma^{j})-D(x,r_{0}\gamma^{j+1})|\leq \epsilon.
\end{equation}
If $\Sigma_{j}$ is $(k+1,\tau)$-independent in $B_{r_{0}\gamma^{j}}(x_{a}^{j})$ with $\tau=\gamma/10$, then by Corollary \ref{quantitative cone splitting corollary}, $u$ is $(k+1,\eta)$-symmetric on $B_{r_{0}\gamma^{j}}(x_{a}^{j})$, however, recall $r=r_{0}\gamma^{l}\leq r_{0}\gamma^{j+1}$, which contradicts $x_{a}^{j}\in S_{\eta,r}^{k}$, therefore we conclude that
\begin{equation}
    \Sigma_{j}\subset B_{r_{0}\gamma^{j+1}/10}(x_{a}^{j}+V_{k})
\end{equation}
for some subspace of dimension $k$. In this case, by a volume estimate, we have $N_{j+1,T}\leq c(n)\gamma^{-k}N_{j,T}$. Since $|T|\leq M$ and note that $N_{0,T}\leq c(n)r_{0}^{-n}=C(n,\lambda,\theta,\Lambda,\eta,\gamma)$, then by an iteration,
\begin{align}
\begin{split}
    N_{l,T}&\leq (c(n)\gamma^{-n})^{M}\cdot (c(n)\gamma^{-k})^{l-M}\cdot N_{0,T}\\
    &\leq c(n)^{l}C\cdot\gamma^{-lk-nM}.
    \end{split}
\end{align}
Therefore $N_{l}\leq c(n)^{l}C\cdot\gamma^{-lk-nM}$, combining \eqref{4.8.1}, we conclude that
\begin{align}
    \begin{split}
           {\rm Vol}(B_{r}(S_{\eta,r}^{k})\cap B_{1}(0))&\leq \omega_{n}N_{l}l^{M}(2r_{0}\gamma^{l})^{n}\\
         &\leq \omega_{n}l^{M}(2r_{0}\gamma^{l})^{n}\cdot c(n)^{l}C\gamma^{-lk-nM}\\
         &\leq C(n,\tau,\lambda,\theta,\Lambda,\eta,\gamma)l^{M}c(n)^{l}\gamma^{l(n-k)}.
    \end{split}
\end{align}
Since for any $\delta>0$, we can choose $\gamma=c(n)^{-2/\delta}$, and then we choose $C_{1}=M^{M}$ such that
\begin{equation}
    l^{M}\leq C_{1}\gamma^{-l\delta/2}.
\end{equation}
Therefore the following holds
\begin{equation}
    c(n)^{l}l^{M}\leq C_{1}\gamma^{-l\delta}.
\end{equation}
Hence,
\begin{equation}
        {\rm Vol}(B_{r}(S_{\eta,r}^{k})\cap B_{1}(0))\leq C(n,\lambda,\theta,\Lambda,\eta,\delta)\gamma^{l(n-k-\delta)}=Cr^{n-k-\delta}.
\end{equation}
And for general case $r_{0}\gamma^{l+1}<r\leq r_{0}\gamma^{l}$ for some $l\geq 1$, we have
\begin{align}
\begin{split}
        {\rm Vol}(B_{r}(S_{\eta,r}^{k})\cap B_{1}(0))&\leq {\rm Vol}(B_{r_{0}\gamma^{l}}(S_{\eta,r_{0}\gamma^{l}}^{k})\cap B_{1}(0))\leq C(n,\tau,\lambda,\theta,\Lambda,\delta)\gamma^{l(n-k-\delta)}\\
      &\leq C\gamma^{-n}r^{n-k-\epsilon}=C(n,\lambda,\theta,\Lambda,\eta,\delta)r^{n-k-\delta}.
      \end{split}
\end{align}
For $1\geq r\geq r_{0}\gamma$, since 
\begin{equation}
    r^{n-k-\delta}\geq r^{n}\geq (r_{0}c(n)^{-2/\delta})^{n},
\end{equation}
and $  {\rm Vol}(B_{1}(0))=\omega_{n}$, then we can choose $C=C(n,\lambda,\theta,\Lambda,\delta,\eta)$ such that $Cr^{n}\geq \omega_{n}$. Finally, we can conclude that for any $0<r\leq 1$, 
\begin{equation}
       {\rm Vol}(B_{r}(S_{\eta,r}^{k})\cap B_{1}(0))\leq Cr^{n-k-\delta}.
\end{equation}
\end{proof}
\subsection{$\epsilon$-regularity and volume estimates}
\indent In the following, we are going to prove the uniform volume estimates for nodal sets. First we introduce $\epsilon$-regularity proposition, since we need the $C^{1}$-regularity of $u$, here we assume that $A$ is not only continuous, but also satisfies the DMO condition. 
\begin{proposition}\label{epsilon regularity in section 4.3}
     Let $u\in W^{1,2}(B_{4}(0))$ be a nonzero weak solution to \eqref{elliptic equation}, with the coefficients satisfying \eqref{unifomly elliptic}, \eqref{continuous coefficient} and \eqref{DMO} in $B_{4}(0)$. For any $0<\delta\leq 1/100$, there exists $r_{0}(n,\lambda,\theta,\omega,\delta)$ and $\epsilon_{0}(n,\lambda,\theta,\omega,\delta)$ such that if $u$ is $(n-1,\epsilon)$-symmetric on $B_{r}(x)$ with $r\leq r_{0}$ and $\epsilon\leq\epsilon_{0}$, then 
     \begin{equation}
         ||\nabla(\hat{u}_{x,r}-l)||_{L^{\infty}(B_{1/4\sqrt{1+\lambda}}(0))}\leq c_{0}\delta,
     \end{equation}
     where $l$ is a linear function with $|\nabla l|\geq c_{0}(n,\lambda,\Lambda)>0$, and therefore $|\nabla u(x)|\neq 0$.
\end{proposition}
\begin{proof}
    By definition, $||\hat{u}_{x,r}-l||_{L^{\infty}(B_{1/\sqrt{1+\lambda}}(0))}\leq\epsilon$, where $\hat{u}_{x,r}=[u(x+ry)-u(x)]/\sqrt{\fint_{B_{1}(0)}|u(x+ry)-u(x)|^{2}dy}$ and $||l||_{L^{2}(\partial B_{1}(0))}\geq c(n,\lambda,\Lambda)>0$. Denote $\hat{A}(y)=A(x+ry)$ and $\hat{u}=\hat{u}_{x,r}$, then $\hat{u}-l$ satisfies
    \begin{equation}
        {\rm div}(\hat{A}(y)\nabla (\hat{u}-l))=-{\rm div}(\hat{A}(y)\nabla l)=-{\rm div}((\hat{A}(y)-\hat{A}(0))\nabla l):=-{\rm div}f,
    \end{equation}
where $f=(\hat{A}(y)-\hat{A}(0))\textbf{L}:=g(ry)$, and $\textbf{L}=\nabla l$ is a constant vector with $c_{1}\geq |L|\geq c_{0}$. It is easy to see that $f$ is continuous and satisfies uniform DMO condition \ref{DMO}, 
  then by the proof of Theorem 1.5 in \cite{DMO}, we have 
    \begin{equation}\label{4.9.1}
        ||\nabla(\hat{u}-l)||_{L^{\infty}(B_{1/4\sqrt{1+\lambda}}(0))}\leq C(||\nabla(\hat{u}-l)||_{L^{1}(B_{1/2\sqrt{1+\lambda}}(0))}+\int_{0}^{1}\frac{\tilde{\omega}_{f}(s)}{s}ds),
        \end{equation}
    where $C$ is a uniform constant, $f(y)=g(ry)$ and $\int_{0}^{1}\frac{\tilde{\omega}_{g}(s)}{s}ds<\infty$. Recall that $\omega_{A}\leq\omega$, then
    \begin{equation}
        \int_{0}^{1}\frac{\tilde{\omega}_{f}(s)}{s}ds=\int_{0}^{r}\frac{\tilde{\omega}_{g}(s)}{s}ds\leq c_{1}\int_{0}^{r}\frac{\tilde{\omega}(s)}{s}\to 0\;\text{as}\;r\to 0.
    \end{equation}
By Hölder inequality and Caccioppoli inequality,
    \begin{equation}\label{4.9.2}
        ||\nabla(\hat{u}-l)||_{L^{1}(B_{1/2\sqrt{1+\lambda}}(0))}\leq C(||\hat{u}-l||_{L^{2}(B_{1/\sqrt{(1+\lambda)}})}+||f||_{L^{2}(B_{1}(0))}).
    \end{equation}
Recall $A$ is continuous, and the modulus of continuity is $\theta$, then ${\rm sup}_{y\in B_{1}(0)}|\hat{A}(y)-\hat{A}(0)|\leq \theta(r)$, hence $||f||_{L^{2}(B_{1}(0))}\leq c_{1}\theta(r)\to 0$ as $r\to 0$. Therefore, combining \eqref{4.9.1} and \eqref{4.9.2}, we can conclude that
\begin{equation}
     ||\nabla(\hat{u}-l)||_{L^{\infty}(B_{1/4\sqrt{1+\lambda}}(0))}\leq C(\epsilon+\theta(r)+\int_{0}^{r}\frac{\tilde{\omega}(s)}{s}ds).
\end{equation}
Since $|\nabla l|=|\textbf{L}|\geq c_{0}(n)$, then we choose $r\leq r_{0}(n,\lambda,\theta,\omega,\delta)$ and $\epsilon\leq \epsilon_{0}(n,\lambda,\theta,\omega,\delta)$ such that $C(\epsilon+\theta(r)+\int_{0}^{r}\frac{\tilde{\omega}(s)}{s}ds)\leq \delta c_{0}(n)$.
\end{proof}
Now we can show the weak volume estimates for critical sets.
\begin{theorem}\label{critical set estimate thm}
     Let $u\in W^{1,2}(B_{4}(0))$ be a nonzero weak solution to \eqref{elliptic equation}, with the coefficients satisfying \eqref{unifomly elliptic}, \eqref{continuous coefficient} and \eqref{DMO} in $B_{4}(0)\subset\mathbb{R}^{n}$. Assume that $u$ satisfies the doubling condition \eqref{doubling for critical set}, then for any $\delta,r\in(0,1)$, there exists a constant $C$, depending only on $n,\lambda,\theta,\omega,\Lambda,\delta$ such that
     \begin{equation}
         {\rm Vol}(B_{r}(C(u))\cap B_{1}(0))\leq Cr^{2-\delta}.
     \end{equation}
\end{theorem}
\begin{proof}[Proof of Theorem \ref{critical set estimate thm}]
    By Proposition \ref{epsilon regularity in section 4.3}, we have
    \begin{equation}
        C(u)\subset S_{\eta,r}^{n-2}
    \end{equation}
    for some $\eta\leq \eta_{0}$ and $r\leq r_{0}$. Therefore by Theorem \ref{singular stratum volume estimate}, 
    \begin{equation}
        {\rm Vol}(B_{r}(C(u))\cap B_{1}(0))\leq{\rm Vol}(B_{r}(S_{\eta_{0},r}^{n-2})\cap B_{1}(0))\leq C(n,\lambda,\theta,\omega,\Lambda,\delta)r^{2-\delta}
    \end{equation}
    for any $0<\delta<1$.
\end{proof}
Next, we use the volume estimates for singular strata, $\epsilon$-regularity and a successful decomposition of $B_{1}$ to prove the uniform volume estimates for nodal sets.
\begin{proof}[Proof of Theorem \ref{nodal set main theorem}]
    Fix $\delta,r_{0}$ as in Proposition \ref{epsilon regularity in section 4.3}, then for any $0<r\leq r_{0}$, note that
we have the following covering of $B_{1}(0)$
\begin{equation}\label{decomposition}
    B_{1}(0)\subset S_{\eta,r}^{n-2}\cup\bigcup_{r\leq 2^{i}r\leq 1/4}(S_{\eta,2^{i+1}r}^{n-2}\setminus S_{\eta,2^{i}r}^{n-2})\cup(B_{1}\setminus S_{\eta,1/4}^{n-2}).
\end{equation}
Next, we will estimate each part that intersects with $Z(u)$. First we claim that if $x\in Z(u)\setminus S_{\eta,h}^{n-2}$, then for any $s\leq h/10$, 
\begin{equation}\label{d0}
      {\rm Vol}(B_{s}[Z(u)\cap B_{h/8\sqrt{1+\lambda}}(x)])\leq C(n,\lambda)sh^{n-1}.
\end{equation}
\indent Since $x\notin S_{\eta,h}^{n-2}$, then by Proposition \ref{epsilon regularity in section 4.3}, there exists some linear function $l$ with $|\nabla l|\geq c_{0}$ such that $||\nabla(\hat{u}-l)||_{L^{\infty}(B_{1/4\sqrt{1+\lambda}}(0))}\leq c_{0}\delta$, where $\hat{u}=\hat{u}_{x,t}$ for some $1>t>h$. Now, we may assume that $\nabla l$ is parallel to $e_{n}$, denote $\hat{u}(y)=\hat{u}(y^{\prime},y_{n})$, then 
\begin{equation}
    |\nabla_{y^{\prime}}\hat{u}|\leq c_{0}\delta\;\text{and}\;|\frac{\partial\hat{u}}{\partial y_{n}}|\geq c_{0}(1-\delta)
\end{equation}
in $B_{1/4\sqrt{1+\lambda}}(0)$. Since $\hat{u}(0^{n-1},0)=0$, then for all $|y^{\prime}|\leq 1/8\sqrt{1+\lambda}<1$, we have $|\hat{u}(y^{\prime},0)|\leq c_{0}\delta$. And we may assume $\frac{\partial\hat{u}}{\partial y_{n}}>0$, therefore
\begin{equation}
    \hat{u}(y^{\prime},0)-\hat{u}(y^{\prime},-4\delta)\geq 4c_{0}\delta(1-\delta)>2c_{0}\delta.
\end{equation}
Then $\hat{u}(y^{\prime},-4\delta)<-c_{0}\delta$, and similarly $\hat{u}(y^{\prime},4\delta)>c_{0}\delta$. Hence, there exists a unique $y_{n}=g(y^{\prime})\in(-4\delta,4\delta)$ such that $\hat{u}(y^{\prime},g(y^{\prime}))=0$ with 
\begin{equation}
    |\nabla_{y^{\prime}}g|=\frac{|\nabla_{y^{\prime}}\hat{u}|}{|\partial \hat{u}/\partial y_{n}|}\leq \frac{\delta}{1-\delta}\leq 2\delta.
\end{equation}
Therefore, 
\begin{equation}
    Z(\hat{u })\cap B_{1/8\sqrt{1+\lambda}}(0)\subset {\rm Graph}(g):=\{(y^{\prime},g(y^{\prime})): |y^{\prime}|\leq 1/8\sqrt{1+\lambda}\}.
\end{equation}
By the definition of $\hat u$, we have 
\begin{equation}
     Z(u)\cap B_{t/8\sqrt{1+\lambda}}(x)\subset x+\{(ty^{\prime},tg(y^{\prime})): |y^{\prime}|\leq 1/8\sqrt{1+\lambda}\}
\end{equation}
is also a $2\delta$-Lipschitz graph.\\
\indent For $h<t$, we conclude that
\begin{equation}
      {\rm Vol}(B_{s}[Z(u)\cap B_{h/8\sqrt{1+\lambda}}(x)])\leq C(n,\lambda)sh^{n-1}
\end{equation}
for any $s\leq h/10$, i.e., the claim holds. By \eqref{decomposition} and Theorem \ref{singular stratum volume estimate}, we obtain the first part estimate in the covering with $s= r/10$
\begin{equation}\label{d1}
      {\rm Vol}(B_{s}(S_{\eta,r}^{n-2}))\leq   {\rm Vol}(B_{r}(S_{\eta,r}^{n-2}))\leq C(n,\lambda,\Lambda,\theta,\eta)r^{2-\eta}.
\end{equation}
And for $x\in S_{\eta,2^{i+1}r}^{n-2}\setminus S_{\eta,2^{i}r}^{n-2}$, by claim, there holds
\begin{equation}\label{42}
      {\rm Vol}(B_{s}[Z(u)\cap B_{2^{i}r/8\sqrt{1+\lambda}}(x)])\leq C(n,\lambda)s(2^{i}r)^{n-1}.
\end{equation}
Hence we can cover $Z_{u}\cap (S_{\eta,2^{i+1}r}^{n-2}\setminus S_{\eta,2^{i}r}^{n-2})$ by a finite number of balls of radius $2^{i}r/8\sqrt{1+\lambda}:=s_{i}/2$ with $s_{i}/4$-disjoint, and let $M_{i}$ be the number of balls, then by volume estimates for singular strata,
\begin{equation}
      {\rm Vol}(B_{s_{i}}(S_{\eta,s_{i}}^{n-2}))\leq C(2^{i+1}r/8\sqrt{1+\lambda})^{2-\eta}.
\end{equation}
Then the disjointness and a direct volume comparison imply
\begin{equation}
    M_{i}\leq C(2^{i}r)^{2-n-\eta}.
\end{equation}
Therefore by \eqref{42}, we obtain that
\begin{equation}\label{d2}
      {\rm Vol}(B_{s}[Z(u)\cap S_{\eta,2^{i+1}r}^{n-2}\setminus S_{\eta,2^{i}r}^{n-2} ])\leq CM_{i}s(2^{i}r)^{n-1}\leq Cs(2^{i}r)^{1-\eta}.
\end{equation}
\indent Recall by \eqref{d0}, 
\begin{equation}\label{d3}
      {\rm Vol}(B_{s} [Z(u)\cap B_{1}\setminus S_{\eta,1/4}^{n-2}])\leq C(n,\lambda)s.
\end{equation}
Note that $s=r/10\leq r_{0}/10$, choosing $\eta=\eta(n,\lambda)<1/10$, therefore by  \eqref{decomposition}, \eqref{d1}, \eqref{d2} and \eqref{d3}, we conclude that
\begin{equation}
      {\rm Vol}(B_{r/10}[Z(u)\cap B_{1}(0)])\leq C(n,\lambda,\Lambda,\theta,\eta)[r^{2-\eta}+\sum_{r\leq 2^{i}r\leq 1/4}r(2^{i}r)^{1-\eta}+r]\leq C(n,\lambda,\Lambda,\theta)r.
\end{equation}
As for $1>r>r_{0}/10$, it is obvious that 
\begin{equation}
      {\rm Vol}(B_{r}\cap [Z(u)\cap B_{1}(0)])\leq {\rm Vol}(B_{1}(0))\leq C(n,\lambda,\Lambda,\theta)r_{0}/10\leq Cr.
\end{equation}
\end{proof}
In the following, we use the Minkowski-type estimates for nodal sets to estimate the sub-level sets. The study of sub-level set is related to the propagation-of-smallness property for elliptic equations, which is a crucial property. First, we show a technical lemma based on an observation in \cite{LM1} .
\begin{lemma}\label{l:zeroes}
     Let $u\in W^{1,2}(B_{4}(0))$ be a nonzero weak solution to \eqref{elliptic equation}, with \eqref{unifomly elliptic} and \eqref{doubling assumption} satisfied in $B_{4}(0)$. Assume that $\fint_{B_{1}(0)}u^{2}=1$, then there exists $r=4e^{-\frac{3b}{4\Lambda}}$ such that if  $b\geq b_{0}(n,\lambda,\Lambda)$, then for any $x_{0}\in B_{1}$ with $B_{r}(x_{0})\cap E_{b}\neq \emptyset$,  there exists $y\in B_{2r}(x_{0})$ such that $u(y)=0$.
\end{lemma}
\begin{remark}
    We will use the doubling condition and Harnack inequality to show the existence of zero point.
\end{remark}
\begin{proof}
    We argue by contradiction, assume that $u$ does not change the sign in $B_{2r}(x_{0})$, without loss of generality, we assume that $u>0$ in it. Then by Harnack inequality, we have
\begin{equation}
    {\rm sup}_{B_{r}(x_{0})}u\leq c_{1}(n,\lambda)~{\rm inf}_{B_{r}(x_{0})}u\leq c_{1}e^{-b}.
\end{equation}

Recall doubling condition \eqref{doubling assumption} tells us
\begin{equation}
    \fint_{B_{2r}(x)}u^{2}\leq 4^{\Lambda}\fint_{B_{r}(x)}u^{2}.
\end{equation}
Hence we take $k$ such that $2\leq 2^{k}r<4$ , by doubling property we obtain
\begin{equation}
     \fint_{B_{2}(x_{0})}u^{2}\leq (\frac{4}{r})^{2\Lambda}\fint_{B_{r}(x_{0})}u^{2}\leq c_{1}^{2}e^{-2b}(\frac{4}{r})^{2\Lambda}
\end{equation}
Now we take $r=4e^{-\frac{3b}{4\Lambda}}$, since $B_{1}\subset B_{2}(x_{0})$, then
\begin{equation}
    \int_{B_{1}(0)}u^{2}\leq c_{4}e^{-b/2}.
\end{equation}
\indent Now we choose $b\geq b_{0}(n,\lambda,\Lambda)$ sufficiently large, then it will contradict the normalization $\fint_{B_{1}(0)}u^{2}=1$. Therefore the assumption does not hold, which means $u$ has a zero in $B_{2r}(x_{0})$. 
\end{proof}
Now we are ready to prove the volume estimates for sub-level sets. 
\begin{proof}[Proof of Theorem \ref{sub-level set theorem}]Let $r=4e^{-\frac{3b}{4\Lambda}}$ with $b\geq b_{0}(n,\lambda,\Lambda)$ as in Lemma \ref{l:zeroes}, and such that $r\leq 1/4$. Recall
\begin{equation}
    E_{b}=S_{e^{-b}}=\{x:|u(x)|\leq e^{-b}\}.
\end{equation}
Next, we show the upper bound of ${\rm Vol}(E_{b}\cap B_{1}(0))$. Assume that $x_{i}\in E_{b}\cap B_{1}(0):=L_{b}$, and we choose a maximal subset such that the distance larger than $r$ for any two points. Therefore we may assume $\cup_{i=1}^{K}B_{2r}(x_{i})\supset L_{b}$, which implies
\begin{equation}
    \Vol(L_{b})\leq c(n)Kr^{n}.
\end{equation}
Next, we need to give an upper bound of $K$, by Lemma \ref{l:zeroes}, 
\begin{equation}
    x_{i}\in E_{b}\;\text{ implies }\;B_{2r}(x_{i})\cap Z_{u}\neq \varnothing.
\end{equation}
Therefore there exists $\bar{x}\in Z_{u}$ such that 
\begin{equation}
    |\bar{x}-x_{i}|<2r.
\end{equation}
Then 
\begin{equation}
    x_{i}\in B_{2r}(\bar{x})\subset B_{2r}(Z_{u}),
\end{equation}
which implies $B_{2r}(x_{i})\subset B_{4r}(Z_{u})$, since $B_{r/2}(x_{i})\cap B_{r/2}(x_{j})=\varnothing$ for each $i\neq j$, we have
\begin{equation}
    Kc(n)r^{n}=\sum_{i=1}^{K} Vol(B_{r/2}(x_{i}))\leq  Vol(B_{4r}(Z_{u})).
\end{equation}
Using the estimate in Theorem \ref{nodal set main theorem}, 
\begin{equation}
      {\rm Vol}(L_{b})\leq Kc(n)r^{n}\leq C(n,\lambda,\Lambda,\theta,\omega)r.
\end{equation}
Recall that $r=4e^{-\frac{3b}{4\Lambda}}$, then we obtain the first desired estimate in Theorem \ref{sub-level set theorem}.\\
\indent For the second estimate, taking $\epsilon={\rm sup}_{E}|u|\leq {\rm sup}_{B_{1}(0)}|u|$, then by doubling property and interior estimate,
    \begin{equation}
        {\rm sup}_{B_{1}(0)}|u|\leq C(n,\lambda,\theta)\fint_{B_{2}(0)}|u|^{2}\leq C_{0}(n,\lambda,\Lambda,\theta).
    \end{equation}
  Hence $\epsilon\leq C_{0}$, then by Theorem \ref{sub-level set theorem}, we have
    \begin{equation}
        {\rm Vol}(\{x\in B_{1}(0):|u(x)|\leq \epsilon\})\leq C(n,\lambda,\Lambda,\theta,\omega)\epsilon^{3/4\Lambda}.
    \end{equation}
    Therefore 
    \begin{equation}
        |E|\leq C({\rm sup}_{E}|u|)^{3/4\Lambda},
    \end{equation}
    and this is the desired result.
\end{proof}

\section{Quantitative uniqueness of tangent map}
In this section, we perform the harmonic approximation more carefully at each scale, and we will use the Dini condition to prove the error at each scale is summable, therefore, we can obtain a quantitative uniqueness result. As we emphasized in Section 2, throughout Section 5 to 7, $D(x,r)$ denotes the doubling index on sphere $D_{s}^{u}(x,r)$ . And we recall the rescaled map here to be
\begin{equation}
    \tilde{u}_{x,r}(y)=\frac{u(x+rA_{x}(y)-u(x))}{\sqrt{\fint_{\partial B_{1}(0)}|u(x+rA_{x}(y)-u(x))|^{2}}}=\frac{u_{x,r}}{||u_{x,r}||},
\end{equation}
where $||z||:=\sqrt{\fint_{\partial B_{1}}|z|^{2}}$. We set $\kappa_{1}=100(1+\lambda)$ and $\Omega(r)=\omega(r)+\int_{0}^{r}\frac{\omega(t)}{t}dt$. Now, we can state the main result in this section.
\begin{theorem}\label{quantitative uniqueness}
    Let $u$ be a solution to ${\rm div}(A\nabla u)=0$, where $A$ is $\omega$-Dini continuous with \eqref{unifomly elliptic}. Assume that $u$ satisfies doubling condition \eqref{doubling assumption}, given $\epsilon\leq c(n,\lambda)$, there exists $r_{0}$ satisfying $C(n,\lambda,\omega)^{\Lambda}[\Omega(\kappa_{1} r_{0})]\leq \epsilon$ such that the following holds. If $|D(x,s)-D(x,r_{1})|\leq \epsilon$ for any $s\in[r_{2},r_{1}]$ with $100r_{2}\leq r_{1}\leq r_{0}$, then there exists a normalized hhp $P_{m}$ such that 
    \begin{equation}
       \sup_{4r_{2}\leq r\leq r_{1}/4} ||\tilde{u}_{x,r}-P_{m}||^{2}\leq 16\epsilon.
    \end{equation}
    Furthermore, we have
    \begin{equation}
        \sup_{4r_{2}\leq r\leq r_{1}/4}\int_{B_{1}(0)}|\tilde{u}_{x,r}-P_{m}|^{2}\leq C(n)\epsilon\;\;\text{and}\;\sup_{4r_{2}\leq r\leq r_{1}/4}||\tilde{u}_{x,r}-P_{m}||_{L^{\infty}(B_{3/4})}\leq C(n,\lambda,\omega)\sqrt{\epsilon}.
    \end{equation}
\end{theorem}
\begin{remark}
    If the Dini parameter $\omega(r)=r^{\alpha}$, we can choose $r_{0}=C(n,\lambda,\alpha)^{-\Lambda}\epsilon^{1/\alpha}$, then the above quantitative uniqueness is the same as Theorem 3.65 in \cite{HJ1}. And the following proof also provides a different approach.
\end{remark}
As a direct corollary, we can obtain an $\epsilon$-regularity result.
\begin{corollary}\label{cor5.2}
    Under the same setting as in Theorem \ref{quantitative uniqueness}, for any $x\in B_{1}$, if $D(x,r)\leq 1+\epsilon$ with $r\leq r_{0}$, then $x\notin C(u)$.
\end{corollary}
\begin{proof}
    Assume that $r_{0}$ satisfies both Theorem \ref{almost monotonicity for doubling index theorem} and Theorem \ref{quantitative uniqueness}. For any $100s\leq r\leq r_{0}$, we have $1-\epsilon \leq D(x,s)\leq D(x,r)+\epsilon\leq 1+2\epsilon$, from the proof of Theorem \ref{quantitative uniqueness}, there exists a linear function $L$ with $||L||=1$ such that $||\tilde{u}_{x,r/4}-L||_{C^{1}(B_{3/4})}\leq C(n,\lambda,\omega)\sqrt{2\epsilon}$, since $|\nabla L|=c(n)$, then $|\nabla \tilde{u}_{x,r/4}(0)|\geq c(n)-C\sqrt{2\epsilon}>0$ if $\epsilon<c(n,\lambda,\omega)$, therefore $x\notin C(u)$.
\end{proof}

To prove Theorem \ref{quantitative uniqueness}, first, we need the following trace theorem.
\begin{lemma}\label{lemma5.2}
    let $v\in C^{1}(B_{3/2})$, then we have
    \begin{equation}
        ||v||^{2}\leq C(n)\int_{B_{3/2}}\left(v^{2}+|\nabla v|^{2}\right).
    \end{equation}
    Recall that  $||v||^2=\fint_{\partial B_1(0)}|v|^2$.
\end{lemma}
\begin{proof}
    The proof is an easy exercise. Fix $\theta\in S^{n-1}$, and let $g(t)=v(t\theta)$, hence there exists $s_{0}\in[t/2,3t/2]$ such that 
    \begin{equation}
        |g(s_{0})|^{2}\leq \frac{1}{t}\int_{t/2}^{3t/2}|g(s)|^{2}ds.
    \end{equation}
Therefore, 
\begin{align}
    \begin{split}
        |g(t)|^{2}&\leq 2|g(s_{0})|^{2}+2[\int_{s_{0}}^{t}|g^{'}(s)|ds]^{2}\\
        &\leq \frac{2}{t}\int_{t/2}^{3t/2}|g(s)|^{2}+2t\int_{s_{0}}^{t}|g^{'}(s)|^{2}\\
        &\leq \frac{2}{t}\int_{t/2}^{3t/2}|g(s)|^{2}+2t\int_{t/2}^{3t/2}|g^{'}(s)|^{2}.
    \end{split}
\end{align}
Next, integrating $\theta$ in $S^{n-1}=\partial B_{1}$, we have
\begin{equation}
    \fint_{\partial B_{1}}|v(t\theta)|^{2}\leq C[t^{-n}\int_{B_{3t/2}\setminus B_{t/2}}|v|^{2}+t^{2-n}\int_{B_{3t/2}\setminus B_{t/2}}|\nabla v|^{2}].
\end{equation}
Let $t=1$, and we obtain the desired estimate.
\end{proof}
Next, we show the $C^{1}$ upper bound for $\tilde{u}_{x,r}$ if $r$ is small enough.
\begin{lemma}\label{lemma5.3}
    Let $w=\tilde{u}_{x,r}$, there exists positive constant $r_{0}$ satisfying $\omega(\kappa r_{0})\leq C(n,\lambda)^{-\Lambda}$ with $\kappa=10(1+\lambda)$ such that if $r\leq r_{0}$, then 
    \begin{equation}
        \fint_{B_{1}(0)}w^{2}\leq C(n)\;\text{and}\;||w||_{C^{1}(B_{1})}\leq C(n,\lambda,\omega)^{\Lambda}.
    \end{equation}
\end{lemma}
\begin{remark}
    Actually, using the same arguments as in Huang-Jiang \cite{HJ1}, we can prove a harmonic approximation for $w$ similar to Lemma \ref{harmonic approximation}, which implies the above lemma directly. Furthermore, one can show the doubling index on sphere is uniformly bounded by $C(n,\lambda)\Lambda$.
\end{remark}
\begin{proof}
    Let $z=u_{x,r}$, it suffices to prove that $||z||\geq c>0$. 
    Let \( {h} \) satisfy the following Dirichlet problem
\begin{equation}
\begin{cases}
\Delta {{h}} = 0 & \text{in } B_{3/2}(0) \\
{h} = {z} & \text{on } \partial B_{3/2}(0).
\end{cases}
\end{equation}
Consider \( v= {z} - {h} \). Since \( \hat{u} \) satisfies the elliptic equation, then
\begin{equation}
-\Delta v = -\Delta {z} = {\rm div}[(B-Id)\nabla z]:={\rm div}f.
\end{equation}
Where $B(y)=a(x)^{-1/2}a(x+rA_{x}(y))a(x)^{-1/2}$ with $|{B}(y)-{B}(0)|\leq C\omega(\sqrt{1+\lambda}r|y|)$. Note that $v=0$ in $\partial B_{3/2}(0)$, then integration by parts and Young inequality imply
\begin{equation}
    \int_{B_{3/2}(0)}|\nabla v|^{2}\leq 2\int_{B_{3/2}(0)}f^{2}.
\end{equation}
Since 
\begin{equation}
    \int_{B_{3/2}(0)}f^{2}\leq [\omega(4\sqrt{(1+\lambda)})]^{2}\int_{B_{3/2}(0)}|\nabla{z}|^{2},
\end{equation}
and by Caccioppoli inequality and doubling assumption, $\int_{B_{3/2}(0)}|\nabla{z}|^{2}\leq C\int_{B_{2}(0)}{z}^{2}\leq C(n,\lambda)^{\Lambda}$, combining Poincare inequality, we can conclude that
\begin{equation}\label{5.13}
|v|_{H^{1}(B_{3/2})} \leq C(n, \lambda)^{\Lambda}\omega(4\sqrt{1+\lambda}r).
\end{equation}
By the last lemma, the trace inequality implies 
\begin{equation}
    ||v||\leq C^{\Lambda}\omega(\kappa_{} r),
\end{equation}
where $\kappa=10(1+\lambda)>1$. Therefore
\begin{align}
    \begin{split}
        ||z||&\geq ||h||-||v||\geq c(n)||h||_{L^{2}(B_{1})}-||v||\\
        &\geq c(n)(||z||_{L^{2}(B_{1})}-||v||_{L^{2}(B_{1})})-||v||\\
        &\geq c(n)-C(n,\lambda)^{\Lambda}\omega(\kappa r)\\
        &\geq c(n)/2>0
    \end{split}
\end{align}
if $r\leq r_{0}$ and $\omega(\kappa r_{0})\leq C^{-\Lambda}$. Then $||w||_{C^{1}(B_{1})}\leq C(n)||z||_{C^{1}(B_{1})}$, by Theorem \ref{DMO regularity} and doubling assumption, we have
\begin{equation}
    ||w||_{C^{1}(B_{1})}\leq C||z||_{L^{2}(B_{1})}\leq C(n,\lambda,\omega)^{\Lambda}.
\end{equation}
This completes the proof.
\end{proof}
Recall that $H_{m}$ denotes the space of homogeneous harmonic polynomial with degree $m$, and classical theory tells us
\begin{equation}
    L^{2}(\partial B_{1})=H_{0}\oplus H_{1}\oplus H_{2}\oplus\cdots,
\end{equation}
here we define 
\begin{equation}
    \Pi_{m}:L^{2}(\partial B_{1})\to H_{m}.
\end{equation}
Fix $x\in B_{1}$, given $\rho>0$ and an integer $m$, we set 
\begin{equation}\label{normal projection}
    Q_{\rho}:=\frac{\Pi_{m}\tilde{u}_{x,\rho}}{||\Pi_{m}\tilde{u}_{x,\rho}||}.
\end{equation}
Next, we prove the following proposition, which plays an important role in showing Theorem \ref{quantitative uniqueness}.
\begin{proposition}\label{projection difference}
     Let $u$ be a solution to ${\rm div}(A\nabla u)=0$, where $A$ is $\omega$-Dini continuous with \eqref{unifomly elliptic}. Assume that $u$ satisfies doubling condition \eqref{doubling assumption}, if $|D^{u}(x,R)-D^{u}(x,R/32)|\leq \epsilon$ with $\epsilon<10^{-3}$ and $C^{\Lambda}\omega(\kappa R)\leq 10^{-3}$, then 
     \begin{equation}
         ||Q_{R/4}-Q_{R/8}||\leq C(n,\lambda)^{\Lambda}\omega(\kappa R),
     \end{equation}
     where the associated integer $m$ is the approximated integer for above doubling index.
    \end{proposition}
\begin{remark}
    Under the Dini condition, the above proposition tells us the error of projection is summable, and the total error will be small if the top scale is small enough.
\end{remark}
\begin{proof}
    Let $w=\tilde{u}_{x,R}$, consider the following Dirichlet problem
\begin{equation}
\begin{cases}
\Delta {{h}} = 0 & \text{in } B_{2}(0) \\
{h} = {w} & \text{on } \partial B_{2}(0).
\end{cases}
\end{equation}
Then we set $v=h-w$, by Lemma \ref{lemma5.2} and Lemma \ref{lemma5.3}, for any $t\in[1/32,1]$, we have
\begin{equation}
    \fint_{\partial B_{1}}|v(t\theta)|^{2}d\theta\leq C\int_{B_{2}}|v|^{2}+|\nabla v|^{2}\leq C(n,\lambda)^{\Lambda}\omega(\kappa R),
\end{equation}
where $\kappa=10(1+\lambda)$. Therefore,
\begin{equation}
    ||w(t\cdot)-h(t\cdot)||\leq C(n,\lambda)^{\Lambda}\omega(\kappa R)
\end{equation}
holds for all $t\in[1/32,1]$. Denote
\begin{equation}
    W_{t}(\theta)=\frac{w(t\theta)}{||w(t\cdot)||},H_{t}(\theta)=\frac{h(t\theta)}{||h(t\cdot)||}.
\end{equation}
 By doubling property, $||w(t\cdot)||\geq C(n,\lambda)^{-\Lambda}>0$, then we have
 \begin{equation}\label{5.24}
     ||W_{t}-H_{t}||\leq C^{\Lambda}\omega(\kappa R).
 \end{equation}
 Since 
 \begin{equation}
     D^{h}(t):={\rm log}_{4}\frac{\fint_{\partial B_{2t}}h^{2}}{\fint_{\partial B_{t}}h^{2}}={\rm log}_{2}\frac{||h(2t\cdot)||}{||h(t\cdot)||},D^{w}(t):={\rm log}_{2}\frac{||w(2t\cdot)||}{||w(t\cdot)||},
 \end{equation}
and $||h(s\cdot)||/||w(s\cdot)||=1+\mu_{s}$ with $|\mu_{s}|\leq C^{\Lambda}\omega(\kappa R)$ for $s\in[1/32,1]$, therefore
 \begin{equation}
     |D^{h}(t)-D^{w}(t)|\leq |{\rm log}_{2}\frac{1+\mu_{2t}}{1+\mu_{t}} |\leq C^{\Lambda}\omega(\kappa R).
 \end{equation}
 Note that $D^{u}(x,tR)=D^{w}(t)$, then we have
 \begin{equation}
     |D^{h}(1)-D^{h}(1/32)|\leq \epsilon+C^{\Lambda}\omega(\kappa R):=\delta_{R}.
 \end{equation}
 Assume that $h=\sum_{k=0}^{\infty}a_{k}P_{k}$, where $P_{k}$ is the normalized hhp of degree $k$. By quantitative uniqueness of tangent map for harmonic function, if $C^{\Lambda}\omega(\kappa R)\leq 10^{-3}$ and $\epsilon<10^{-3}$, then $\delta_{R}<1/100$, hence there exists some integer $m$ such that for any $t\in[1/8,1/4]$
 \begin{equation}\label{5.28}
   ||\Pi_{m}H_{t}||^{2}\geq (1-6\delta_{R})\;\text{and}\;  \fint_{\partial B_{1}}|H_{t}-P_{m}|^{2}\leq 9\delta_{R}.
 \end{equation}
And we can also write $h(r\theta)=\sum_{k=0}^{\infty}a_{k}r^{k}p_{k}(\theta)$, then $\Pi_{m}h=p_{m}$, and $P_{m}=p_{m}$ on $\partial B_{1}$. Hence by the definition of $H_{t}(\theta)$, we can rewrite the above quantitative uniqueness as
\begin{equation}
    ||H_{t}-\frac{\Pi_{m}H_{t}}{||\Pi_{m}H_{t}||}||\leq 3\sqrt{\delta_{R}}
\end{equation}
for any $t\in[1/8,1/4]$. Now fix $m$, note that $||\Pi_{m}(H_{t}-p_{m})||\leq ||H_{t}-p_{m}||\leq 3\delta_{R}$, then since $\delta_{R}<1/100$, we have $||\Pi_{m}H_{t}||\geq 1-3/100>3/4$ for any $t\in[1/8,1/4]$. Since $u_{x,tR}=W_{t}$ on $\partial B_{1}$, recall that 
\begin{equation}
    ||\Pi_{m}W_{t}-\Pi_{m}H_{t}||\leq ||W_{t}-H_{t}||\leq C^{\Lambda}\omega(\kappa R).
\end{equation}
Then
\begin{equation}
    ||\frac{\Pi_{m}W_{t}}{||\Pi_{m}W_{t}||}-\frac{\Pi_{m}H_{t}}{||\Pi_{m}H_{t}||}||\leq C^{\Lambda}\omega(\kappa R).
\end{equation}
By the definition of $Q_{\rho}$, we can conclude that
\begin{equation}
    ||Q_{tR}-p_{m}||\leq C^{\Lambda}\omega(\kappa R)
\end{equation}
for all $t\in[1/8,1/4]$. Therefore,
\begin{equation}
    ||Q_{R/4}-Q_{R/8}||\leq C^{\Lambda}\omega(\kappa R).
\end{equation}
\end{proof}
Now, we are ready to prove the quantitative uniqueness theorem under the Dini setting.
\begin{proof}[Proof of Theorem \ref{quantitative uniqueness}]
    Taking $\rho=r_{1}/4$ and $s_{j}=2^{-j}\rho$ for $j=0,1,2,...$, and we only consider $j$ such that $s_{j}\geq 8r_{2}$. By pinching condition $|D^{u}(x,t)-D^{u}(x,r_{1})|\leq \epsilon$ for any $t\in[t_{2},t_{1}]$, then Proposition \ref{projection difference} implies that for any $s\in[s_{j}/2,s_{j}]$, there exists $P_{j}\in H_{m_{j}}$ with $||P_{j}||=1$ such that 
    \begin{equation}
        ||\tilde{u}_{x,s}-P_{j}||\leq \epsilon+ C^{\Lambda}\omega(\kappa r_{j}).
    \end{equation}
Then we claim that $m_{j}=m_{j+1}$, if not, we may assume $m_{j}\neq m_{j+1}$ for some $j$, then 
\begin{equation}
    2=||P_{j}-P_{j+1}||\leq ||\tilde{u}_{x,j+1}-P_{j}||+||\tilde{u}_{x,j+1}-P_{j+1}||\leq 2\epsilon+C^{\Lambda}[\omega(\kappa r_{j})+\omega(\kappa r_{j+1})],
\end{equation}
    which is a contradiction if $\epsilon$ and $r_{1}\leq r_{0}$ is small enough. Therefore, $m_{j}\equiv m$ for all $j$ such that $s_{j}\geq 8r_{2}$. Hence we can define the normalization projection $Q_{s}$ as \eqref{normal projection} for all $s\in[4r_{2},r_{1}]$ with the same $d$. By Proposition \ref{projection difference}, we have 
    \begin{equation}
        ||Q_{s_{j}}-Q_{s_{j+1}}||\leq C^{\lambda}\omega(\kappa s_{j}).
    \end{equation}
Then
\begin{align}
\begin{split}
    ||Q_{s_{N}}-Q_{s_{0}}||&\leq C^{\Lambda}\sum_{j=0}^{N-1}||Q_{s_{j+1}}-Q_{s_{j}}||\leq C^{\Lambda}\sum_{j=0}
^{N-1}\omega(\kappa s_{j})\\
&\leq C^{\Lambda}\sum_{j=0}^{N-1}\int_{\kappa s_{j}}^{2\kappa s_{j}}\frac{\omega(s)}{s}ds\leq C^{\Lambda}\int_{0}^{2\kappa s_{0}}\frac{\omega(s)}{s}ds
\end{split}
\end{align}
    Denote $P_{m}=Q_{s_{0}}$ and recall that $s_{0}\leq r_{1}$, then for any $s\in[4r_{2},r_{1}/4]$, we have
    \begin{equation}
        \sup_{4r_{2}\leq r\leq r_{1}/4}||Q_{r}-P_{m}||\leq C^{\Lambda}\int_{0}^{\kappa r_{1}}\frac{\omega(s)}{s}ds:=C^{\Lambda}\omega_{1}(\kappa r_{1}).
    \end{equation}
Finally, for any $r\in[4r_{2},r_{1}/4]$, we can write 
\begin{equation}
    \tilde{u}_{x,r}=a_{r}Q_{r}+R_{r},
\end{equation}
where $a_{r}=||\Pi_{m}\tilde{u}_{x,r}||$, by \eqref{5.24}, taking $R=8r$ and $t=1/8$, then $||\tilde{u}_{x,r}-H_{1/8}||\leq C^{\Lambda}\omega(8\kappa r)$, then $||\Pi_{m}\tilde{u}_{x,r}-\Pi_{m}H_{1/8}||\leq C^{\Lambda}\omega(8\kappa r)$, and by \eqref{5.28}, we have
\begin{equation}
    a_{r}^{2}=||\Pi_{m}\tilde{u}_{x,r}||^{2}\geq 1-6\delta_{8r}-C^{\Lambda}\omega(8\kappa r)^{2}\geq 1-\tau_{r},
\end{equation}
where $\tau_{r}=6\epsilon+C^{\Lambda}\omega(8\kappa r )$, then 
\begin{align}
\begin{split}
    ||\tilde{u}_{x,r}-P_{m}||&\leq ||\tilde{u}_{x,r}-Q_{r}||+||Q_{r}-P_{m}||\\
    &\leq \sqrt{(1-a_{r})^{2}+1-a_{r}^{2}}+C\omega_{1}(\kappa r_{1})\\
    &\leq \sqrt{2(1-a_{r}^{2})}+C\omega_{1}(\kappa r_{1})\\
    &\leq \sqrt{2\tau_{r}}+C^{\Lambda}\omega_{1}(\kappa r_{1})\leq 4\sqrt{\epsilon}
    \end{split}
\end{align}
    if $r_{1}\leq r_{0}$ and $10C^{\Lambda}[\Omega(8\kappa r_{0})]\leq \epsilon$, where $\Omega(r)=\omega(r)+\omega_{1}(r)$. Furthermore, for any $r\in [4r_{2},r_{1}/4]$, taking $h_{r}$ as the harmonic approximation of $\tilde{u}_{x,r}$ in $B_{1}$, then we have 
    \begin{align}
    \begin{split}
        ||\tilde{u}_{x,r}-P_{m}||_{L^{2}(B_{1})}&\leq ||\tilde{u}_{x,r}-h_{r}||_{L^{2}(B_{1})}+||h_{r}-P_{m}||_{L^{2}(B_{1})}\\
        &\leq C\omega(\kappa r)+c(n)||h_{r}-P_{m}||\\
        &\leq \sqrt{\epsilon}+c(n)||\tilde{u}_{x,r}-P_{m}||\leq C(n)\sqrt{\epsilon}.
        \end{split}
    \end{align}
    For $C^{1}$-estimate, since 
    \begin{equation}
        -\Delta(\tilde{u}_{x,r}-P_{m})={\rm div}((B-I)\nabla \tilde{u}_{x,r}):={\rm div}f,
    \end{equation}
    then by elliptic estimate we have
    \begin{equation}
        ||\tilde{u}_{x,r}-P_{m}||_{L^{\infty}(B_{3/4})}\leq C(n,\lambda,\omega)(||\tilde{u}_{x,r}-P_{m}||_{L^{1}(B_{7/8})}+||f||_{L^{n+1}(B_{7/8})}).
    \end{equation}
    By Theorem \ref{DMO regularity} and Lemma \ref{lemma5.3}, $||f||_{L^{n+1}(B_{7/8})}\leq \omega(\kappa r)||\nabla \tilde{u}_{x,r}||_{L^{\infty}(B_{7/8})}\leq C\omega(\kappa r)||\tilde{u}_{x,r}||_{L^{2}(B_{1})}\leq C(n,\lambda,\omega)\omega(\kappa r)$. By the choice of $r_{0}$, we conclude that
    \begin{equation}
         ||\tilde{u}_{x,r}-P_{m}||_{L^{\infty}(B_{3/4})}\leq C(n,\lambda,\omega)\sqrt{\epsilon}.
    \end{equation}
    Hence we obtain the desired quantitative uniqueness result.
\end{proof}

\section{Quantitative cone splitting principle}
In this section, we generalize the quantitative cone splitting principle to Dini case. In roughly speaking, the main idea is that two pinched points will give an almost invariant direction. First, we recall two lemmas for hhp, which was proved by Naber-Valtorta.
\begin{lemma}[\cite{NV} Proposition 3.28 ]\label{lemma 6.1}
    Let $P_{m}$ is a hhp of degree $m$ in $\mathbb{R}^{n}$ with $||P_{m}||=1$, there exists $\epsilon_{0}(n)$ such that if 
    \begin{equation}
        ||\partial_{i}P_{m}||^{2}\leq \epsilon||\nabla P_{m}||^{2}
    \end{equation}
    for $i=1,2,...,k$ with $\epsilon\leq \epsilon_{0}/m^{2n}$ and $k\leq n-2$, then there exist normalized $Q,R\in H_{m}$ such that $\partial_{i}Q=0$ for $i=1,2,...,k$ and
    \begin{equation}
        P_{m}=\sqrt{1-\delta^{2}}Q+\delta  R,
    \end{equation}
    where $\delta\leq C(n)m^{n}\sqrt{\epsilon}$.
\end{lemma}
And if a hhp $P$ has $(n-1)$ almost invariant direction, then $P$ is linear.
\begin{lemma}[\cite{NV} Lemma 3.23]\label{lemma6.2}
    Let $P_{m}$ be a nonconstant hhp in $\mathbb{R}^{n}$ of degree $m$ such that
    \begin{equation}
       || \partial_{i}P_{m}||^{2}\leq \epsilon||\nabla P_{m}||^{2}
    \end{equation}
    for $i=1,2,...,n-1$ with $\epsilon\leq 1/2n^{2}$, then $m=1$.
\end{lemma}
To state the quantitative cone splitting principle, we need the following notion. Let
\begin{equation}
    \mathcal{P}_{m,\delta,r}^{u}(x)=\{y\in B_{r}(x):|D^{u}(y,s)-m|\leq \delta\;\text{holds for all}\;s\in[10^{-2}r,C_{1}(\lambda)\Lambda r]\},
\end{equation}
where $C_{1}=10^{3}(1+\lambda)$ and we call it the $(m,\delta)$-pinched set in $B_{r}(x)$. Next, we state the quantitative cone splitting principle, which is an improvement of Proposition \ref{cone splitting prop}. 
\begin{theorem}\label{cone splitting theorem}
     Let $u$ be a nonzero solution to ${\rm div}(A\nabla u)=0$ in $B_{2}(0)$, where $A$ is $\omega$-Dini continuous with \eqref{unifomly elliptic}. Assume that $u$ satisfies doubling condition \eqref{doubling assumption}, fix $\tau<c(n,\lambda)$, then there exists $\epsilon_{0}=C(n,\lambda,\omega,\tau)^{-\Lambda}$ such that if $\epsilon\leq \epsilon_{0}$, $r\leq r_{1}$ with $C_{1}\Lambda r_{1}\leq r_{0}$ as in Theorem \ref{quantitative uniqueness}, the following holds. Assume that $\mathcal{P}:=\mathcal{P}_{m,\epsilon,r}^{u}(x)$ is $(k,\tau)$ independent in $B_{r}(x)\subset B_{1}(0)$, then there exists a normalized hhp $Q$ of degree $m$, which is $k$-symmetric such that
     \begin{equation}
         \sup_{r/10\leq s\leq 10\Lambda r}||\tilde{u}_{y,s}-Q||\leq C(n,\lambda,\omega,\tau)^{\Lambda}\sqrt{\epsilon}\leq \epsilon^{1/3}.
     \end{equation}
     for each $y\in\mathcal{P}_{m,\delta,r}^{u}(x)$.
\end{theorem}
\begin{proof}
    Denote $R=Kr$, where $K=20\sqrt{1+\lambda}$ is to be determined, and both $r$ and $R$ are in the interval of quantitative uniqueness. Let $w=\tilde{u}_{y,R}$, fix $r_{0}$ as in Theorem \ref{quantitative uniqueness} and $C_{1}\Lambda r_{1}\leq r_{0}$, then there exists $P_{y}\in H_{m}$ with $||P_{y}||=1$ such that
\begin{equation}\label{6.5}
    ||w-P_{y}||_{L^{\infty}(B_{3/4})}\leq C(n,\lambda,\omega)\sqrt{\epsilon}.
\end{equation}
Taking another $z\in \mathcal{P}$, and set $b=A_{y}^{-1}\frac{z-y}{R}$ and $M:=M_{yz}:=A_{y}^{-1}A_{z}$, where $A_{y}=\sqrt{A(y)}$, note that $|A_{y}-A_{z}|\leq C(n)|A(y)-A(z)|\leq C(n)\omega(|y-z|)$, then $|M-I|\leq C\omega(2r)<1/10$ if $r$ is small enough. By the choice of $K$, then $|b|\leq 1/10$. Since $z=y+RA_{y}b$, then
\begin{equation}
    z+\frac{R}{2}A_{z}\theta=y+RA_{y}(b+\frac{1}{2}M\theta).
\end{equation}
We define 
\begin{equation}
    E(\theta)=w(b+\frac{1}{2}M\theta)-w(b)\;\text{and}\;G(\theta)=P_{y}(b+\frac{1}{2}M\theta)-P_{y}(b).
\end{equation}
Recall that $w=\tilde{u}_{y,R}$, then $u_{z,R/2}=E/||E||$. Since $\theta\in\partial B_{1}$, then  $b+\frac{1}{2}M\theta\in B_{3/4}$. By \eqref{6.5}, we have
\begin{equation}
    ||E-G||\leq C\sqrt{\epsilon}.
\end{equation}
Next we show that $||G||\geq C^{-\Lambda}>0$, therefore 
\begin{equation}\label{6.9}
    ||u_{z,R/2}-\frac{G}{||G||}||\leq C^{\Lambda}\sqrt{\epsilon}.
\end{equation}
Using Taylor expansion
\begin{equation}\label{Taylor expansion}
    G(\theta)=\sum_{j=0}^{m-1}\frac{1}{j!\cdot2^{m-j}}[(b\cdot \nabla)^{j}P_{y}(M\theta)],
\end{equation}
then $ \Pi_{m}G=\frac{1}{2^{m}}\Pi_{m}(P_{y}\circ M)$. Since 
\begin{equation}\label{6.12}
    ||\Pi_{m}(P_{y}\circ M)-P_{y}||=||\Pi_{m}(P_{y}\circ M-P_{y})||\leq C|M-I|\leq 1/10,
\end{equation}
and $m\leq C(n,\lambda)\Lambda$, then
\begin{equation}
    ||G||\geq ||\Pi_{m}G||\geq C^{-\Lambda}||\Pi_{m}(P_{y}\circ M)||\geq C^{-\Lambda}/2.
\end{equation}
By Theorem \ref{quantitative uniqueness}, there exists $P_{z}\in H_{m}$ such that $||u_{z,R/2}-P_{z}||\leq 4\sqrt{\epsilon}$, combining \eqref{6.9}, we have
\begin{equation}\label{6.13}
    ||P_{z}-\frac{G}{||G||}||\leq C^{\Lambda}\sqrt{\epsilon}.
\end{equation}
\indent \textbf{Claim 1:} $||b\cdot \nabla P_{y}||\leq C^{\Lambda}\sqrt{\epsilon}$. Note that for each homogeneous polynomial $p_{m}$, $p_{m}|_{\partial B_{1}}=\sum_{j=0}^{[m/2]}q_{m-2j}$, where $q_{k}\in H_{k}$, therefore by \eqref{Taylor expansion}
\begin{equation}
    \Pi_{m-1}G=\frac{1}{2^{m-1}}\Pi_{m-1}[(b\cdot \nabla P_{y})\circ M].
\end{equation}
By \eqref{6.13}, $||\Pi_{m-1}G||\leq C^{\Lambda}\sqrt{\epsilon}||G||\leq C^{\Lambda}\sqrt{\epsilon}$. Similar to \eqref{6.12} and note that $b\cdot \nabla P_{y}\in H_{m-1}$, 
\begin{equation}
    ||\Pi_{m-1}[(b\cdot \nabla P_{y})\circ M]||\geq ||b\cdot \nabla P_{y}||-||\Pi_{m-1}[(b\cdot \nabla P_{y})\circ M-b\cdot \nabla P_{y}]||\geq ||b\cdot \nabla P_{y}||/2.
\end{equation}
    Therefore $||b\cdot \nabla P_{y}||\leq C^{\Lambda}\sqrt{\epsilon}$, i.e., Cliam 1 holds. \\
\indent Recall the definition of $b$, we have $||\nabla P_{y}\cdot A_{y}^{-1}(\frac{z-y}{r})||\leq C^{\Lambda}\sqrt{\epsilon}$. Since $\mathcal{P}$ is $(k,\tau)$-independent in $B_{r}(x)$, then we can choose $\{z_{0},z_{1},...,z_{k}\}\subset \mathcal{P}$ such that 
\begin{equation}
{\rm dist}(z_{i},{\rm Aff}(z_{0},z_{1},...,z_{i-1}))\geq \tau r
\end{equation}
for $i=1,2,...,k$. Now let $y=z_{0},P=P_{z_{0}}$ and define $v_{i}=A_{y}^{-1}(\frac{z_{i}-z_{0}}{r})$, then $|v_{i}|\geq \tau/\sqrt{1+\lambda}$ and 
\begin{equation}
    ||\partial_{v_{i}}P_{0}||\leq C^{\Lambda}\sqrt{\epsilon}.
\end{equation}
By Gram-Schmidt orthogonalization, we can obtain an orthogonal basis $\{e_{1},e_{2},...,e_{k}\}$ of dimension $k$ from $\{v_{1},v_{2},...,v_{k}\}$, which satisfies
\begin{equation}
    ||\partial_{e_{i}}P||\leq C(n,\lambda,\tau,\omega)^{\Lambda}\sqrt{\epsilon}||\nabla P||,
\end{equation}
where we use $||\nabla P||=C(n,m)>1$. Denote $\delta=C^{2\Lambda}\epsilon$, therefore by Lemma \ref{lemma 6.1}, if $\delta\leq c(n)(C\Lambda)^{-2n}$, i.e., $\epsilon\leq \epsilon_{0}\leq C(n,\lambda,\tau,\omega)^{-\Lambda}$, then there exist normalized $Q,R\in H_{m}$ such that $\partial_{e_{i}}Q=0$ for $i=1,2,...,k$ with
\begin{equation}
    P=\sqrt{1-\eta^{2}}Q+\eta R.
\end{equation}
Therefore $||P-Q||\leq 2\eta$, where $\eta\leq C(n)\Lambda^{n}\sqrt{\delta}$, then $||P-Q||\leq C^{\Lambda}\epsilon^{1/2}$. Recall that $||u_{z,s}-P_{z}||\leq 4\sqrt{\epsilon}$ for any $s\in [r/10,10\Lambda r]$ and $P_{z_{0}}=P$, to finish the whole proof, it suffices to prove that $||P_{z}-P_{z_{0}}||\leq C^{\Lambda}\sqrt{\epsilon}$ if $r$ is small enough .\\
\indent \textbf{Claim 2:} For any $y,z\in\mathcal{P}$, and denote $P_{y},P_{z}\in H_{m}$ be corresponding tangent maps in Theorem \ref{quantitative uniqueness}, if $r\leq r_{1}$ with $C_{1}\Lambda r_{1}\leq r_{0}$, then $||P_{y}-P_{z}||\leq C\sqrt{\epsilon}$. First, recall \eqref{Taylor expansion} implies
\begin{equation}
    \Pi_{m}G=\frac{1}{2^{m}}\Pi_{m}(P_{y}\circ M).
\end{equation}
By \eqref{6.13} and projection onto $H_{m}$, we have 
\begin{equation}
    ||\alpha T-P_{z}||\leq C^{\Lambda}\sqrt{\epsilon}.
\end{equation}
where $T:=\Pi_{m}G/||\Pi_{m}G||$ and $\alpha=||\Pi_{m}G||/||G||$. Since
\begin{equation}
    |\alpha-1|=|||\alpha T||-||P_{z}|||\leq ||\alpha T-P_{z}||\leq C^{\Lambda}\sqrt{\epsilon}.
\end{equation}
Therefore $||T-P_{z}||\leq C^{\Lambda}\sqrt{\epsilon}$. Similar to \eqref{6.12}, $||\Pi_{m}(P_{y}\circ M)-P_{y}||\leq C^{\Lambda}\omega(|y-z|)$, then $||T-P_{y}||\leq C^{\Lambda}\omega(2r)$. By the choice of $r_{0}$ in Theorem \ref{quantitative uniqueness} and $C_{1}\Lambda r_{1}\leq r_{0}$, we have $\omega(2r_{1})\leq \sqrt{\epsilon}$, then we can conclude that $||P_{y}-P_{z}||\leq C^{\Lambda}\sqrt{\epsilon}$.
\end{proof}
As a corollary, we can obtain a Lipchitz structure for pinched set. For any $y\in \mathcal{P}_{m,\epsilon,r}(x_{0})$, we define the $m$-pinched scale as
\begin{equation}
    r_{y}:=\sup_{s}\{0\leq s\leq r: D(y,s)\leq m-\epsilon\}.
\end{equation}
\begin{corollary}\label{cor6.4}
Under the same setting as in Theorem \ref{cone splitting theorem} with $m\geq 2,\tau\leq c(n,\lambda)$, assume that $E\subset \mathcal{P}_{m,\epsilon,r}(x_{0})$ satisfying for any $y,z\in E$, ${\rm max}\{r_{y},r_{z}\}\leq |y-z|\leq r/100$ and fix $r_{1}$ as in Theorem \ref{cone splitting theorem}. Then there exists $\epsilon_{0}=C(n,\lambda,\omega,\tau)^{-\Lambda}$ such that if $\epsilon\leq \epsilon_{0}$, $r\leq r_{1}$, we have a universal plane $V$ of dimension $\leq n-2$, which is independent of  the pinched points and scales, such that $E\subset {\rm Graph}(f)$, where $f:V\to V^{\perp}$ is a Lipschitz map with ${\rm Lip}(f)\leq 1/10$.
\end{corollary}
\begin{proof}
    For any $y,z\in\mathcal{P}$, taking $r=|y-z|$ in the proof of Theorem \ref{cone splitting theorem}, then $b=A_{y}^{-1}\frac{z-y}{20\sqrt{1+\lambda}|z-y|}\in B_{1/20}$. And by Claim 1, we have $||b\cdot \nabla P_{y}||\leq C^{\Lambda}\sqrt{\epsilon}$. Let $\hat{P}_{y}(\xi)=P_{y}(A_{y}^{-1}\xi)$, then for vector $h=\frac{z-y}{20\sqrt{1+\lambda}|z-y|}$, $\partial_{h}\hat{P}_{y}(\xi)=b\cdot \nabla P_{y}(A_{y}^{-1}\xi)$, hence by doubling property and coordinate transformation,
    \begin{equation}\label{6.24}
        ||\partial_{z-y}\hat{P}_{y}||\leq C^{\Lambda}\sqrt{\epsilon}|z-y|.
    \end{equation}
\indent We may assume that $E$ is $(k,\tau)$-independent in $B_{r}(x_{0})$ for some $k$, and we fix $z_{0}\in E$, by Claim 2, $||P-P_{y}||\leq C^{\Lambda}\sqrt{\epsilon}$, where $P=P_{z_{0}}$. Since $z_{0}\in E$, by  the proof of Theorem \ref{cone splitting theorem}, we can find a $k$-symmetric $Q\in H_{m}$ such that $||P-Q||\leq C(n)\Lambda^{n}\sqrt{\epsilon}$, therefore $||P_{y}-Q||\leq C^{\Lambda}\sqrt{\epsilon}$ for any $y\in\mathcal{P}$. Since $P_{y},Q\in H_{m}$, then doubling property and the interior estimate imply
\begin{equation}
    ||P_{y}-Q||_{C^{1}(B_{\sqrt{1+\lambda}})}\leq C^{\Lambda}\sqrt{\epsilon}.
\end{equation}
For any $z\in\mathcal{P}$, set $\hat{Q}_{z}(\xi)=Q(A_{z}^{-1}\xi)$, then we have
\begin{equation}
    ||\hat{P}_{y}-\hat{Q}_{z}||_{C^{1}(B_{1})}\leq C^{\Lambda}(\sqrt{\epsilon}+\omega(r)).
\end{equation}
Therefore, we can choose a common $\hat{Q}(\xi)=Q(A^{-1}\xi)$ with $(1+\lambda)^{-1/2}I\leq A\leq (1+\lambda)^{1/2}I$  such that for any vector $v$,
\begin{equation}
    ||\partial_{v}\hat{P}_{y}-\partial_{v}\hat{Q}||\leq C^{\Lambda}(\sqrt{\epsilon}+\omega(r))|v|.
\end{equation}
Taking $v=|z-y|$, then combining \eqref{6.24}, we have $||\partial_{z-y}\hat{Q}||\leq C^{\Lambda}\sqrt{\epsilon}|z-y|$ since $\omega(r)\leq \omega(r_1)\leq \sqrt{\epsilon}$.\\
\indent Now we define the following quadratic form with respect to $\hat{Q}$,
\begin{equation}
    \hat{B}(v,w):=\frac{\langle \partial_{v}\hat{Q},\partial_{w}\hat{Q}\rangle }{||\nabla \hat{Q}||^{2}},
\end{equation}
where the inner product $\langle f,g\rangle:=\fint_{\partial B_{1}}fg$. Note that $\hat{B}$ is semi-positive definite, we assume its eigenvalues $0\leq \lambda_{1}\leq...\leq \lambda_{n-1}\leq \lambda_{n}$, and the corresponding orthogonal eigenvectors $f_{1},...,f_{n}$. Now we define $V:={\rm span}\{f_{1},f_{2},...,f_{n-2}\}$, especially, if $Q$ is $(n-2)$-symmetric with respect to $W$, then $V=AW$.\\
\indent Set $v=z-y=\sum_{i=1}^{n}\alpha_{i}f_{i}$, then ${\rm dist}(v,V)=\sqrt{\alpha_{n-1}^{2}+\alpha_{n}^{2}}$. Note that 
$\hat{B}(v,v)=\sum_{i=1}^{n}\lambda_{i}\alpha_{i}^{2}\geq \lambda_{n-1}{\rm dist}(v,V)^{2}$, since $||Q||=1$ and $||\nabla \hat{Q}||^{2}\geq c(n,\lambda)>0$, recall $||\partial_{v}\hat{Q}||\leq C^{\Lambda}\sqrt{\epsilon}|v|$, then
\begin{equation}\label{6.29}
    \lambda_{n-1}{\rm dist}(z-y,V)^{2}\leq \hat{B}(v,v)\leq C^{2\Lambda}\epsilon|z-y|^{2}.
\end{equation}
\indent Finally, we give a uniform lower bound of $\lambda_{n-1}$. Similarly, we denote define
\begin{equation}
    B(v,w):=\frac{\langle \partial_{v}Q,\partial_{w}Q\rangle }{||\nabla {Q}||^{2}}.
\end{equation}
    And we denote its eigenvalues as $0\leq \mu_{1}\leq...\leq \mu_{n}$ and its corresponding orthogonal eigenvectors as $e_{1},...,e_{n}$, one can verify that $\lambda_{n-1}\geq C(n,\lambda)^{-\Lambda}\mu_{n-1}$. Hence it suffices to give a uniform lower bound of $\mu_{n-1}$. Actually, we claim that $\mu_{n-1}>1/2n^{2}$. If not, $\mu_{i}\leq\mu_{n-1}\leq 1/2n^{2}$ for $i\leq n-1$, then Lemma \ref{lemma6.2} implies $m=1$, which contradicts the assumption $m\geq 2$, hence the claim is true. So we conclude that $\lambda_{n-1}\geq C(n,\lambda)^{-\Lambda}$. Combining \eqref{6.29}, we have
    \begin{equation}
        {\rm dist}(z-y,V)\leq C^{\Lambda}\sqrt{\epsilon}|z-y|\leq |z-y|/100
    \end{equation}
    if $\epsilon\leq C^{-\Lambda}$. The above inequality gives a Lipschitz map $f:V\to V^{\perp}$ such that $E\subset {\rm Graph}(f)$ with ${\rm Lip}(f)\leq 1/10$.
\end{proof}
In the last part of this section, we point out the location of critical set near an almost $(n-2)$-symmetric point.
\begin{lemma}\label{lemma6.5}
       Let $u$ be a solution to ${\rm div}(A\nabla u)=0$ in $B_{2}(0)$, where $A$ is $\omega$-Dini continuous with \eqref{unifomly elliptic}. Assume that $u$ satisfies doubling condition \eqref{doubling assumption}, fix $\tau=c(n,\lambda),r_{1}$ as in Theorem \ref{cone splitting theorem}, then there exists $\epsilon_{0}=C(n,\lambda,\omega,\tau)^{-\Lambda}$ such that if $\epsilon\leq \epsilon_{0}$, $r\leq r_{1}$, the following holds. For any $x\in B_{1}$, if there exists a $(n-2)$-symmetric hhp $P$ with respect to $V$ such that $||u_{x,2\sqrt{1+\lambda}r}-P||\leq \epsilon$, then 
       \begin{equation}
           C(u)\cap B_{r}(x)\subset B_{\tau r}(x+V_{x}),
       \end{equation}
       where $V_{x}=A_{x}V$.
\end{lemma}
\begin{remark}
    For $\mathcal{P}=\mathcal{P}_{m,\epsilon,r}^{u}(x_{0})$ the same as in Theorem \ref{cone splitting theorem}, if $\mathcal{P}$ is $(n-2,\tau_{1})$-independent in $B_{r}(x_{0})$ and $r\leq r_{0}$, note that $P_{x}\approx P_{y}$ and $A_{x}\approx A_{y}$ for any $x,y\in \mathcal{P}$, hence one can choose a common $V$ as in Corollary \ref{cor6.4} for Lemma \ref{lemma6.5}, which is independent of pinched points and scales.
\end{remark}
\begin{proof}
    Since $||u_{x,2\sqrt{1+\lambda}r}-P||\leq \epsilon$, by the proof in Theorem \ref{quantitative uniqueness}, we have $||u_{x,2\sqrt{1+\lambda}r}-P||_{C^{1}(B_{3/4})}\leq C(n,\lambda,\omega)\epsilon$. We may assume that $V={\rm span}\{e_{1},e_{2},...,e_{n-2}\}$, and $P(x)=P(x_{n-1},x_{n-2})$ with degree $m\leq C\Lambda$, then 
    \begin{equation}\label{6.34}
        |\nabla P(x)|=\sqrt{2}m\cdot {\rm dist}(x,V)^{m-1}.
    \end{equation}
    If there exists $z\in C(u)\cap B_{r}(x)$ with ${\rm dist}(z-x,V_{x})>\tau r$, then let $y=A_{x}^{-1}(\frac{z-x}{2\sqrt{1+\lambda}r})\in B_{1/2}$, we have ${\rm dist}(y,V)\geq c(\lambda){\rm dist}(z-x,A_{x}V)/r\geq c(\lambda)\tau$. Hence by \eqref{6.34}, $|\nabla P(y)|\geq [c(\lambda)\tau]^{C\Lambda}$. However, $||\nabla u_{x,2\sqrt{1+\lambda}r}-\nabla P||_{L^{\infty}(B_{3/4})}\leq C\epsilon$, especially, $|\nabla u_{x,2\sqrt{1+\lambda}r}(y)|\geq [c\tau]^{C\Lambda}-C\epsilon>0$ if $\epsilon<C(n,\lambda,\omega)^{-\Lambda}$. Hence $\nabla u(z)\neq 0$, which contradicts $z\in C(u)$. 
     \end{proof}

\section{Volume estimates for critical sets}
In this section, we apply the quantitative uniqueness and cone splitting techniques to prove the Minkowski estimates for critical sets.
\subsection{Interior case}
\indent First, we fix the constants. Taking $\tau=c(n,\lambda),\kappa=C_{1}=100(1+\lambda)$, and choosing common $\epsilon=C(n,\lambda,\omega)^{-\Lambda}$ and $C_{1}\Lambda r_{1}\leq r_{0}$ with $\Omega(\kappa r_{0})\leq [c(n,\lambda,\omega)\epsilon/2]^{C(n,\lambda)\Lambda}\leq C(n,\lambda,\omega)^{-\Lambda^{2}}$ as in Theorem \ref{sphere almost monotonicity}, Corollary \ref{cor3.9}, Corollary \ref{cor6.4} and Lemma \ref{lemma6.5}. In the following, we always work on $B_{r}(x_{0})$ with $r\leq r_{1}$, and we will rescale it to $B_{1}(0)$ for the convenience. For later use of inductive covering, we introduce the concept of $(m,\delta)$-good ball, i.e., we say $B_{r}(x)$ is an $(m,K)$-good ball if $D^{u}(y,s)\leq m+\epsilon$ for any $y\in B_{r}(x)$ and $s\leq K r$.\\
\indent To prove the volume estimates, we need a key covering lemma as in \cite{NV} and \cite{HJ1}, and the proof also follow their methods. The main idea for proving it is a successful inductive covering argument introduced in \cite{NV}, and the main ingredients are Corollary \ref{cor6.4} and Lemma \ref{lemma6.5}, which are due to quantitative uniqueness and cone splitting principle. 
\begin{lemma}\label{key covering lemma}
   Under the same setting as in Theorem \ref{interior critical theorem}, fix $r>0$ and assume that $B_{1}(0)$ is a $(m,C_{1}(\lambda)\Lambda)$-good ball with $m\geq 2$, then we have the following covering
    \begin{equation}
        C(u)\cap B_{1}(0)\subset \bigcup_{j}B_{t_{j}}(x_{j})\;\text{with}\;\sum_{j}t_{j}^{n-2}\leq C(n,\lambda)\cdot (\Lambda/\epsilon)^{n},
    \end{equation}
    where $x_{j}\in C(u )$ and each $B_{t_{j}}(x_{j})$ is either an $(m-1,C_{1}(\lambda)\Lambda)$-good ball or the radius satisfies $t_{j}=r$.
\end{lemma}
\begin{proof}
    First, we may assume that $r\leq (C_{1}\Lambda)^{-1}$, if not, we can choose a direct Vitali covering as the desired covering. Recall the definition of $m$-pinched scale, i.e, 
    \begin{equation}
        \tilde{r}_{x}=\sup\{0\leq s\leq 1:D(x,s)\leq m-\epsilon\}\;\text{and}\;r_{x}:=\max\{\tilde{r}_{x},r\}\geq r.
    \end{equation}
    Then we can decompose $C(u)$ in $B_{1}(0)$ into the good part and bad part
    \begin{equation}
        C(u)\cap B_{1}(0)=C_{g}\cup C_{b},
    \end{equation}
    where 
    \begin{equation}
        C_{g}:=\{x\in C(u)\cap B_{1}:r_{y}\geq r_{x}/7 \;\text{for any}\;y\in B_{5r_{x}}(x)\} \;\text{and}\; C_{b}=[C(u)\cap B_{1}]\setminus C_{g}.
    \end{equation}
    Here $r_{y}\geq r_{x}/7$ means that $D(x,7r_{y})\geq m-\epsilon$ and recall that now $1$ is the scale for almost monotonicity with error $\epsilon/2$, then $D(y,r_{x}/7)\leq D(y,r_{y})+\epsilon/2\leq  m-\epsilon/2$. Next, we can follow \cite{HJ1} to cover $C_{g}$ and $C_{b}$, now we give a proof sketch here.\\
\indent  \textbf{Step 1: The covering of  $C_{g}$.} First, we choose a Vitali covering of $C_{g}$, i.e., 
    \begin{equation}
        C_{g}\subset \bigcup_{i}B_{5r_{i}}(x_{i})\; \text{with}\;B_{r_{i}}(x_{i})\cap B_{r_{j}}(x_{j})=\emptyset \;\text{for any}\;i\neq j,
  \end{equation}
  where $x_{i}\in C_{g}$ and $r_{i}:=r_{x_{i}}$. If $r_{i}>r$, then $D(y,r_{i}/7 )\leq m-\epsilon/2$ for any $y\in B_{5r_{i}}(x_{i})$. By Corollary \ref{cor3.9} with error $\epsilon/2$, we have $D(y,\epsilon r_{i}/56)\leq m-1+\epsilon/2$. Then we can cover $B_{r_{i}}(x_{i})$ with the ball with radius $\epsilon r_{i}/56C_{1}\Lambda$, and the number of such ball for $B_{r_{i}}(x_{i})$ is bounded by $C(n,\lambda)\cdot(\Lambda/\epsilon)^{n}$. Besides, by almost monotonicity, each new ball is an $(m-1,C_{1}\Lambda)$-good ball. Therefore for the desired covering of $C_{g}$, it suffices to show that $\sum_{i}r_{i}^{n-2}\leq C(n)$.\\
  \indent Denote $G:=\{x_{i}\in C_{a}:r_{i}>1/100\}$, since the cover is disjoint, then $\sum_{x_{i}\in G}r_{i}^{n-2}\leq C(n)\cdot 1^{n-2}=C(n)$. Hence it remains to estimate the sum of points with radius $\leq 1/100$. We may set $E:=C_{g}\setminus G$, and $E=\bigcup_{a\in \mathcal{A}}E_{a}$, where for any $x_{i},x_{j}\in E_{a}$, $|x_{i}-x_{j}|\leq 1/100$. By disjointeness, $|\mathcal{A}|\leq C(n)$, therefore it suffices to prove that $\sum_{x_{i}\in E_{a}}r_{i}^{n-2}\leq C(n)$. Since $r_{i}+r_{j}\leq |x_{i}-x_{j}|\leq 1/100$, then $x_{i},x_{j}\in \mathcal{P}_{m,\epsilon,1}(0)$. By Corollary \ref{cor6.4}, there exists a Lipschitz map $f:V\to V^{\perp}$ with ${\rm dim}V\leq n-2$ and ${\rm Lip}(f)\leq 1/10$ such that $E_{a}\in {\rm Graph}(f)$. Therefore, by disjointness and Lipschitz structure, we have $\sum_{x_{i}\in E_{a}}r_{i}^{n-2}\leq C(n)$.\\
  \indent \textbf{Step 2: The covering of $C_{b}$.} For any $y\in C_{b}$, by definition and iteration, there exists some $x\in C_{g}$ such that
  \begin{equation}
      |x-y|\leq 5\sum_{j=0}7^{-j}r_{y}\leq 6r_{y}\; \text{and}\;r_{x}< r_{y}/7.
  \end{equation}
    We may assume that $x\in B_{5r_{i}}(x_{i})$, then $r_{x}\geq r_{i}/7$, which means $r_{i}\leq r_{y}$ and $|y-x_{i}|\leq 11r_{y}$. Denote $s_{y}:=\min_{i}|y-x_{i}|/55$, then $s_{y}\leq r_{y}/5$, and we also choose a Vitali covering of the remaining part, i.e.,
    \begin{equation}
        C_{b}\setminus \bigcup_{i}B_{5r_{i}}(x_{i})\subset \bigcup_{j}B_{s_{j}}(y_{j})\; \text{with}\;B_{s_{j}}(y_{j})\cap B_{s_{k}}(y_{k})=\emptyset \;\text{for any}\;j\neq k,
    \end{equation}
    where $y_{j}\in C_{b}$ and $s_{j}:=s_{y_{j}}$. Following the same further discussions in \cite{HJ1}, one can also verify the doubling index drop condition and $\sum_{j}s_{j}^{n-2}\leq C(n)$, here we omit them.
\end{proof}
\indent Based on the above covering lemma, now we are ready to prove Theorem \ref{interior critical theorem}.
\begin{proof}[Proof of Theorem \ref{interior critical theorem}]
    Fix the constants $\tau,C_{1},\epsilon,\kappa, r_{0}$ as before, and especially, we set $r_{0}$ such that $\Omega(\kappa r_{0})=C(n,\lambda,\omega)^{-\Lambda^{2}}$ and $r_{1}=r_{0}/C_{1}\Lambda$. And we can cover $B_{1}(0)$ by several $B_{r_{1}}(x_{i})$ with the number of such balls $\leq C(n)r_{1}^{-n}$. Note that if we set $m=[C(n,\lambda)\Lambda]+1$, then each ball is an $(m,C_{1}\Lambda)$-good ball. For $r>r_{1}$, applying Lemma \ref{key covering lemma} to each $B_{r_{1}}(x_{i})$, then we have
    \begin{equation}
        C(u)\cap B_{1}\subset \bigcup_{j}B_{t_{1,j}}(x_{1,j})\;\text{with}\;\sum_{j}t_{1,j}^{n-2}\leq C(n)r_{1}^{-n}\cdot C(n,\lambda,\omega)^{\Lambda},
    \end{equation}
    where $x_{1,j}\in C(u)$ and each $B_{t_{1,j}}(x_{1,j})$ is an $(m-1,C_{1}\Lambda)$-good ball. Iterating $m-1$ times, we obtain
    \begin{equation}
        C(u)\cap B_{1}\subset \bigcup_{j}B_{t_{j}}(x_{j})\;\text{with}\;\sum_{j}t_{j}^{n-2}\leq C(n)r_{1}^{-n}\cdot C(n,\lambda,\omega)^{(m-1)\Lambda}\leq  C(n,\lambda,\omega)^{\Lambda^{2}}r_{1}^{-n}.
    \end{equation}
    Finally, each $B_{t_{j}}(x_{j})$ is a $(1,C_{1}\Lambda)$-good ball or $t_{j}=r$. By Corollary \ref{cor5.2}, if $D(x_{j},t_{j})\leq 1+\epsilon$, then $x_{j}\notin C(u)$, which is a contradiction. Hence $t_{j}=r$ for all $j$, and the number of these balls is bounded by $C(n,\lambda,\omega)^{\Lambda^{2}}\cdot r_{1}^{-n}r^{2-n}$, then we have
    \begin{equation}
        {\rm Vol}(B_{r}(C(u)\cap B_{1}))\leq C(n,\lambda,\omega)^{\Lambda^{2}}\cdot  r_{1}^{-n}r^{2-n}\cdot \alpha_{n}(2r)^{n}\leq C^{\Lambda^{2}}r_{0}^{-n}\cdot r^{2}.
    \end{equation}
\end{proof}
\subsection{Boundary case} In this subsection, we recall the boundary results of \cite{KZ4}, and combining our interior volume estimate, we can obtain Theorem \ref{boundary critiical theorem}. First, we need to introduce the concept of $C^{1}$-Dini domain, roughly speaking, $\partial D$ is locally a $C^{1}$-Dini graph.
\begin{definition}\label{def7.2}
    We say a connected domain $D\subset \mathbb{R}^{n}$ is a $C^{1}$-Dini domain with Dini parameter $\omega$ if for any point $x_{0}\in \partial D$, there exists a coordinate system $x=(x^{\prime},x_{n})$ with $x^{\prime}\in\mathbb{R}^{n-1}$ and $x_{n}\in \mathbb{R}$ such that $x_{0}=(0^{n-1},0)$ in this coordinate and there are a ball $B$ centered at $x_{0}$ and a Lipschitz function $g$ in $\mathbb{R}^{n-1}$ satisfying \\
    \indent \quad $(1)$ $||\nabla g||_{L^{\infty}(\mathbb{R}^{n-1})}\leq C$ for some constant $C>0$;\\
    \indent \quad $(2)$ $|\nabla g(x^{\prime})-\nabla g(y^{\prime})|\leq \omega(|x^{\prime}-y^{\prime}|)$ for any $x^{\prime},y^{\prime}\in\mathbb{R}^{n-1}$;\\
    \indent \quad $(3)$ $D\cap B=\{(x^{\prime},x_{n})\in B: x_{n}>g(x^{\prime})\}$.\\
    Especially, if $\omega(r)\leq C_{\alpha}r^{\alpha}$ for some $\alpha\in(0,1]$, we say $D$ is a $C^{1,\alpha}$ domain.
\end{definition}
Based on the main result in \cite{HJ1}, i.e, under the doubling assumption, for elliptic equation ${\rm div}(A\nabla u)=0$ with H\"{o}lder coefficients, the critical set satisfies $\mathcal{H}^{n-2}(C(u)\cap B_{1}(0))\leq C<\infty$, where $C$ is a universal constant, Kenig-Zhao \cite{KZ4} proved the following theorem.
\begin{theorem}[\cite{KZ4} Theorem 1.1]\label{thm7.3}
Let $D\subset \mathbb{R}^{n}$ be a $C^{1,\alpha}$ domain with constant $C_{\alpha}$ such that $0\in \partial D$. Assume that $u$ is a non-zero harmonic function in $D\cap B_{5R}(0)$, satisfying the boundary condition $u=0$ in $\partial D\cap B_{5R}(0)$. Then we have the Hausdorff estimate for critical set
    \begin{equation}
        \mathcal{H}^{n-2}(C(u)\cap \overline{D}\cap B_{R}(0))\leq C<\infty,
    \end{equation}
   where $C$ depends on $n,\alpha,C_{\alpha},R$ and the upper bound of frequency function in $B_{5R}(0)$.
\end{theorem}
\indent Here, we give a short outline of the proof of this theorem, and we refer interested readers to \cite{KZ4} for more details.
\begin{proof}[Sketch of proof]
    \textbf{Step 1: Reduction to an interior elliptic equation}. Assume that $B_{0}$ be the ball centered at the origin, for $D$ is a $C^{1}$-Dini domain, by definition
    \begin{equation}
        D\cap B_{0}=\{(x^{\prime},x_{n})\in\mathbb{R}^{n-1}\times \mathbb{R}:x_{n}>g(x^{\prime})\}.
    \end{equation}
Define the map
\begin{equation}
    G:(x^{\prime},s)\in \mathbb{R}^{n-1}\times \mathbb{R}_{+}\to (x^{\prime},g(x^{\prime}))+s\cdot \rho_{s}*\vec{n}(x^{\prime})\in\mathbb{R}^{n},
\end{equation}
where $0\leq \rho\in C_{c}(\mathbb{R}^{n-1})$ with ${\rm supp}(\rho)\subset B_{1}(0^{n-1})$ and $\int_{\mathbb{R}^{n-1}}\rho(y)dy=1$, and $\vec{n}$ is the unit normal vector to $\partial D$ at $(x^{\prime},g(x^{\prime}))$, i.e.,
\begin{equation}
    \vec{n}(x^{\prime})=(\frac{-\nabla g(x^{\prime})}{\sqrt{1+|\nabla g(x^{\prime})|^{2}}},\frac{1}{\sqrt{1+|\nabla g(x^{\prime})|^{2}}}).
\end{equation}
One can check that if $u$ satisfies $\Delta u=0$ in $D\cap B_{0}$, then $\tilde{u}(x^{\prime},x_{n}):=u\circ G(x^{'},x_{n})$ satisfies ${\rm div}(A(x)\nabla \tilde{u})=0$ in $\mathbb{R}_{+}^{n}\cap B$, where $x=(x^{'},x_{n})$, $A(x)=|{\rm det}(DG)|DG(x)^{-1}(DG(x)^{-1})^{t}$, $B$ is a slightly small ball than $B_{0}$ with ${\rm det}(DG(x))\approx 1$ in $\mathbb{R}_{+}^{n}\cap B$. Since $u=0$ on $\partial D\cap B_{0}$, then $\tilde{u}=0$ in $\partial \mathbb{R}_{+}^{n}\cap B$. Consider its odd reflection across $\mathbb{R}_{+}^{n}$, and we still denote it as $\tilde{u}$, then we have
\begin{equation}
    {\rm div}(\tilde{A}(x)\nabla \tilde{u})=0\;\text{in both}\;\mathbb{R}_{-}^{n}\;\text{and}\;\mathbb{R}_{+}^{n}.
\end{equation}
Since 
\begin{equation}
    \lim_{x_{n}\to 0+}\tilde{A}(x^{\prime},x_{n})\nabla \tilde{u}(x^{\prime},x_{n})\cdot (-e_{n})+\lim_{x_{n}\to 0-}\tilde{A}(x^{\prime},x_{n})\nabla \tilde{u}(x^{\prime},x_{n})\cdot (e_{n})=0,
\end{equation}
then one can show that 
\begin{equation}
    {\rm div}(\tilde{A}(x)\nabla \tilde{u})=0\;\text{in entire ball}\; B\subset\mathbb{R}^{n}.
\end{equation}
    Note that the modulus of continuity of $\tilde{A}$ is controlled by $\omega$, i.e, $\tilde{A}(x)$ is $\omega$-Dini continuous. Especially, if $D$ is a $C^{1,\alpha}$ domain, then $\tilde{A}(x)$ is $\alpha$-H\"{o}lder continuous.\\
 \indent   \textbf{Step 2: Verifying the doubling assumption.} To apply the result of \cite{HJ1}, one need to check that $\tilde{u}$ satisfies the doubling assumption in $B$. By the construction of $\tilde{u}$, it suffices to verify that $\tilde{u}$ satisfies the doubling in $B\cap \mathbb{R}_{+}^{n}$. Next, for any $x\in\mathbb{R}^{n}$ and $r>r$, we define
 \begin{equation}
   N(x,r):=\frac{rD(x,r)}{H(x,r)}:=\frac{r\int_{B_{r}(x)\cap D}|\nabla u|^{2}}{\int_{\partial B_{r}(x)\cap D}|u|^{2}}\;\text{and}\;  N_{C}(x,r):=\frac{rD(x,r)}{H_{C}(x,r)}:=\frac{r\int_{B_{r}(x)\cap D}|\nabla u|^{2}}{\int_{\partial B_{r}(x)\cap D}|u-u(x)|^{2}},
 \end{equation}
 where $N(x,r)$ is called the \textit{standard frequency function} and $N_{C}(x,r)$ is called the \textit{normalized frequency function}. Denote $N_{x}(r)$ be the boundary modified frequency as in \cite{KZ1}, the discussions in Section 2 of \cite{KZ4} imply that if $N(0,5R)\leq N$, then $N_{0}(4R)\leq CN$. Then, by Lemma 5.26 in \cite{KZ4}, they proved for $C^{1}$-Dini domain $D$ and harmonic function $u$ that $N_{C}(x,r)\leq C(n,R,N)$ holds for any $x\in D\cap B_{R/40}(0)$ and $r\leq R/8$. Once we have the upper bound of normalized frequency function, if the radius is smaller than the critical scale, i.e., $2s\leq 2r\leq r_{cs}(x)$, since
 \begin{equation}
     \frac{d}{ds}({\rm log}\frac{H_{C}(x,s)}{s^{n-1}})=\frac{2N_{C}(x,s)}{s}+{\rm Err}_{s},
 \end{equation}
 where $|{\rm Err}_{s}|\leq C/s$, then integrating the above inequality, we obtain that
 \begin{equation}
     \frac{H_{C}(x,2s)}{H_{C}(x,s)}\leq 2^{n-1+C(N,R)}.
 \end{equation}
 Furthermore, for $2r\leq r_{cs}(x)$, we have
 \begin{equation}
     \frac{\int_{B_{2r}(x)\cap D}|u-u(x)|^{2}}{\int_{B_{r}(x)\cap D}|u-u(x)|^{2}}=\frac{2\int_{0}^{r}H_{C}(x,2s)ds}{\int_{0}^{r}H_{C}(x,s)ds}\leq 2^{n+C(N,R)}.
 \end{equation}
 \indent On the other hand, when ${\rm dist}(x,\partial D)\leq r\leq 2r\leq R/8$, since Lemma 5.7 in \cite{KZ4} tells us for any $x\in D\cap B_{3R/40}(0)$ and ${\rm dist}(x,\partial D)\leq r<R/8$, 
\begin{equation}
    c_{1}\fint_{B_{r}(\tilde{x})}u^{2}\leq \frac{1}{r^{n}}\int_{B_{r}(x)\cap D}|u-u(x)|^{2}\leq c_{2}\fint_{B_{r}(\tilde{x})}u^{2},
\end{equation}
where $c_{1},c_{2}$ are universal constants and $\tilde{x}\in\partial D$ with $|x-\tilde{x}|={\rm dist}(x,\partial D)$. Therefore
\begin{equation}
    \frac{\int_{B_{2r}(x)\cap D}|u-u(x)|^{2}}{\int_{B_{r}(x)\cap D}|u-u(x)|^{2}}\leq C\frac{\int_{B_{2r}(\tilde{x})}u^{2}}{\int_{B_{r}(\tilde{x})}u^{2}}\leq 2^{n+C(N,R)},
\end{equation}
where in the last inequality, we use the uniform upper bound of frequency at boundary point in \cite{KZ1}. Since $r_{cs}(x)\geq 72{\rm dist}(x,\partial D)$, then we can conclude that for any $x\in D\cap B_{R/40}(0)$ and $2r\leq R/8$,
\begin{equation}\label{7.24}
         \frac{\int_{B_{2r}(x)\cap D}|u-u(x)|^{2}}{\int_{B_{r}(x)\cap D}|u-u(x)|^{2}}\leq 2^{n+C(N,R)}.
\end{equation}
    Recall that $\tilde{u}=u\circ G$ with ${\rm det}( DG(x))\approx 1$ in $B\cap \mathbb{R}_{+}^{n}$, if we take $B_{0}=B_{2r}(x)$, since $B$ is slightly small than $B_{0}$ with the radius of $B$ only depends on $R,\omega$. By above estimate \eqref{7.24}, for any ball $b\subset B$,
    \begin{equation}
          \frac{\int_{b}|\tilde{u}-\tilde{u}(x)|^{2}}{\int_{b/2}|\tilde{u}-\tilde{u}(x)|^{2}}\leq C(n,N,R,\omega).
    \end{equation}
   We remark that all the arguments in Step 1 and Step 2 are carried out under the case of $C^{1}$-Dini domain.\\
    \indent \textbf{Step 3: Applying the interior volume estimate.} Since the doubling assumption holds for any $b\in B$. By the main volume estimate in \cite{HJ1}, we have
    \begin{equation}
        \mathcal{H}^{n-2}(C(\tilde{u})\cap B/2)\leq C(n,R,N,\alpha,C_{\alpha}).
    \end{equation}
Since ${\rm det}( DG(x))\approx 1$ in $B\cap \mathbb{R}_{+}^{n}$, then 
\begin{equation}
    \mathcal{H}^{n-2}(C(u)\cap G(\overline{\mathbb{R}_{+}^{n}})\cap B/2)\approx \mathcal{H}^{n-2}(C(\tilde{u})\cap \overline{\mathbb{R}_{+}^{n}}\cap B/2)\leq C(n,R,N,\alpha,C_{\alpha}).
\end{equation}
Since $\overline{D}$ is compact, then there are finitely many maps $G_{i}$ and $B_{i}$ such that
\begin{equation}
    \overline{D}\cap B_{R/100}(0)\subset \bigcup_{i=1}^{M}G_{i}(\overline{\mathbb{R}_{+}^{n}}\cap B_{i}/2),
\end{equation}
where $M$ depends only on $R,\alpha,C_{\alpha}$. Then we can conclude that
\begin{equation}
    \mathcal{H}^{n-2}(C(u)\cap B_{R/100}(0))\leq \sum_{i=1}^{M}\mathcal{H}^{n-2}(C(u)\cap G_{i}(\overline{\mathbb{R}_{+}^{n}}\cap B_{i}/2))\leq C(n,R,N,\alpha,C_{\alpha}).
\end{equation}
Finally, Kenig-Zhao pointed out in Remark 1.2 of \cite{KZ4} that one can improve this estimate to $C(u)\cap B_{R}(0)$.
\end{proof}
\indent Throughout out the prof of Theorem \ref{thm7.3}, the only place where the $C^{1\alpha}$-domain condition is used is the interior volume estimate for critical set in \cite{HJ1} in Step 3. Therefore, as mentioned in \cite{KZ4}, if the volume estimate holds for elliptic equation ${\rm div}(A\nabla u)=0$ with Dini coefficients, we can obtain similar volume estimate for boundary critical set. In other words, Theorem \ref{interior critical theorem} implies Theorem \ref{boundary critiical theorem}.

\appendix
\section{}

\setcounter{equation}{0}
\renewcommand{\theequation}{A\arabic{equation}}
\renewcommand{\theproposition}{A\arabic{proposition}}
Here, we introduce a classical example in quasiconformal geometry, which can serve as a counterexample to Theorem \ref{nodal set main theorem} in the case of continuous coefficients.
\begin{proposition}\label{Prop A.1}
    Let $B_{2}(0)\subset\mathbb{R}^{2}$, there exist a continuous coefficient matrix $A$, $u\in W^{1,2}(B_{2}(0))$, $u\not\equiv 0$ and positive constants $\lambda$ and $\Lambda$ such that:\\
\indent \quad $(1)$ $A(x)$ is uniformly elliptic, i.e., $(1+\lambda)^{-1}I\leq A(x)\leq (1+\lambda)I$;\\
\indent \quad$(2)$ $u$ is a weak solution to ${\rm div}(A\nabla u)=0$ in $B_{2}(0)$;\\
\quad \indent \quad$(3)$ doubling property holds for any $B_{2r}(x)\subset B_{2}(0)$, i.e., $\fint_{B_{2r}(x)}u^{2}\leq 4^{\Lambda}\fint_{B_{r}(x)}u^{2}$;\\ 
\indent \quad $(4)$ the nodal set has infinite length, i.e. $\mathcal{H}^{n-1}(Z(u)\cap B_{1}(0))=\infty$.
\end{proposition}
First, we recall some basic notions in quasiconformal geometry. 
\begin{definition}
$(1)$ Let $\Omega,\Omega'\subset\mathbb C$ be domains. We say an orientation-preserving homeomorphism $f:\Omega\longrightarrow\Omega'$ is \emph{$K$-quasiconformal}, where $1\leq K<\infty$, if $f\in W_{\mathrm{loc}}^{1,2}(\Omega)$ and $  \|\mu_f\|_{L^\infty(\Omega)}
    \leq\frac{K-1}{K+1}$, where $\mu_{f}:=f_{\bar{z}}/f_{z}$ is the Beltrami coefficient.\\
\indent $(2)$ Let $\mathbb{D}\subset \mathbb{C}$ be the unit disk. We say $\Gamma$ is a asymptotically conformal quasicircle if there exists a global quasiconformal homeomorphism $f:\mathbb{C}\to\mathbb{C}$ such that $\Gamma=f(\partial \mathbb{D})$ and ${\rm ess\,sup}_{1-t<|z|<1+t}|\mu_{f}(z)|\to 0$ as $t\to 0$.
\end{definition}
Next, we recall a standard existence result in quasiconformal geometry, which is important for the counterexample. We refer the readers to Section 1 of \cite{Badger} for this and more related result.
\begin{lemma}\label{lemma A.3}
    There is a Jordan curve $\Gamma\subset\mathbb{C}$ such that $\Gamma$ is an asymptotically conformal quasicircle with ${\rm dim}_{\mathcal{H}}(\Gamma)=1$ and $\mathcal{H}^{1}(\Gamma)=\infty$.
\end{lemma}
\indent As a consequence, by a compactness argument, there exists $p\in\Gamma$ such that for any $r>0$,
\begin{equation}\label{A.1}
    \mathcal{H}^{1}(\Gamma\cap B_{r}(p))=\infty.
\end{equation}
\begin{proof}[Proof of Proposition \ref{Prop A.1}]
    \indent Taking $\Gamma$ as in Lemma \ref{lemma A.3} and $p\in\Gamma$ as above, by definition of asymptotically conformal quasicircle, assume that $\Gamma=f(\partial \mathbb{D})$, but $f$ may not be conformal. Let $\Omega$ be the bounded component of $\mathbb{C}\setminus \Gamma$, and we choose $g:\mathbb{D}\to\Omega$ be a Riemann map. Then by the equivalence in \cite{Dy}, the classical characterization of asymptotically conformal quasicircle gives
    \begin{equation}\label{asymptotic conformal}
        \lim_{t\to 0}\sup_{1-t<|z|<1}(1-|z|)|\frac{g^{\prime\prime}(z)}{g^{\prime}(z)}|=0.
    \end{equation}
\indent Since $\Gamma$ is a Jordan curve, by Caratheodory extension theorem, $g:\overline{\mathbb{D}}\to \overline{\Omega}$ is homeomorphism, and we may assume that $z_{0}\in\partial \mathbb{D}$ such that $g(z_{0})=p$. Next, denote $z^{*}=1/\bar{z}$, then we define
\[G(z)=
\left\{
\begin{aligned}
    &g(z),\quad \quad\quad\quad\quad\quad\quad\;\;\quad|z|\leq 1,\\
    &g(z^{*})+g^{\prime}(z^{*})(z-z^{*}),\quad 1<|z|<1+\delta.
\end{aligned}
\right.
\]
    By a direct calculation, for $1<|z|<1+\delta$,
    \begin{equation}\label{A.3}
        \mu_{G}(z):=\frac{G_{\bar{z}}}{G_{z}}=-(z^{*})^{2}(z-z^{*})\frac{g^{\prime\prime}(z^{*})}{g^{\prime}(z^{*})}.
    \end{equation}
Therefore by \eqref{asymptotic conformal} and  \eqref{A.3},
\begin{equation}
    \lim_{\delta\to 0}\sup_{1<|z|<1+\delta}|\mu_{G}(z)|=2   \lim_{\delta\to 0}\sup_{(1+\delta)^{-1}<|z^{*}|<1}(1-|z^{*}|)|\frac{g^{\prime\prime}(z^{*})}{g^{\prime}(z^{*})}|=0.
\end{equation}
\indent Note that $\mu_{G}=0$ in $\mathbb{D}$, set $\mu_{G}=0$ on $\partial \mathbb{D}$, we have $\mu_{G}$ is continuous across the circle. Choosing $\delta$ small enough, we may assume that $||\mu_{G}||_{L^{\infty}}\leq k<1$ in $\{z:|z|< 1+\delta\}$. Denote $K=(1+k)/(1-k)$, then both $G$ and $G^{-1}$ are $K$-quasiconformal.
Now we define 
\begin{equation}
    M(z):=i\cdot \frac{1-\bar{z}_{0}z}{1+\bar{z}_{0}z},\;H:=M\circ G^{-1},\;u={\rm Im}\,H.
\end{equation}
 Since $M$ is conformal, then $|\mu_{H}|=|\mu_{G^{-1}}|=|\mu_{G}|$. Now identifying the complex plane $\mathbb{C}$ as $\mathbb{R}^{2}$, then $H$ can define on $B_{2r}(p)\subset \mathbb{R}^{2}$ for some small $r$. After translation and rescaling, we may assume that $u,H$ are defined on $B_{2}(0)\subset \mathbb{C}=\mathbb{R}^{2}$. One can check that
\begin{equation}
    Z(u)\cap B_{1}(0)=\Gamma_{\rm loc}\cap B_{1}(0),
\end{equation}
where $\Gamma_{\rm loc}$ is the translated and rescaled local arc of $\Gamma$ near $p$. Therefore by the choice of $p$, we have
\begin{equation}
    \mathcal{H}^{1}(Z(u)\cap B_{1}(0))=\infty,
\end{equation}
i.e., $(4)$ holds. Recall that $u={\rm Im}\,H,u\not \equiv 0$, hence the classical quasiconformal theory implies that $u\in W^{1,2}(B_{2}(0))$ and is a solution to ${\rm div}(A\nabla u)=0$ in the weak sense, where
\[
A(z)=\frac{1}{1-|\mu_{H}(z)|^{2}}
\begin{pmatrix}
    |1-\mu_{H}(z)|^{2} & -2{\rm Im}\,\mu_{H}(z)\\
    -2{\rm Im}\,\mu_{H}(z) & |1+\mu_{H}(z)|^{2}
\end{pmatrix}
    .
\]
Note that the eigenvalues of $A(z)$ are 
\begin{equation}
    \frac{1-|\mu_{H}(z)|}{1+|\mu_{H}(z)|}\;\text{and}\; \frac{1+|\mu_{H}(z)|}{1-|\mu_{H}(z)|},
\end{equation}
since $||\mu_{H}||_{L^{\infty}(B_{2}(0))}\leq k<1$, then $A$ is uniformly elliptic with $\lambda=2k/(1-k)$, hence $(1)$ and $(2)$ also hold. Finally, we check the doubling property $(3)$.\\
\indent Since $H$ is a $K$-quasiconformal homeomorphism, the classical local quasisymmetry implies that for any $y\in B_{2r}(x)\in B_{2}(0),z\in \partial B_{r/2}(x)$, we have 
\begin{equation}
    |H(y)-H(x)|\leq \eta_{K}(4)|H(z)-H(x)|.
\end{equation}
Denote $  M(x,r):=\sup_{y\in {B_{r}(x)}}|H(y)-H(x)|$ and $m(x,r):=\inf_{z\in\partial B_{r}(x)}|H(z)-H(x)|$. Then 
\begin{equation}
    M(x,2r)\leq C(K)m(x,r/2)\;\text{for any}\,B_{2r}(x)\subset B_{2}(0).
\end{equation}
Note that $B_{m(x,r/2)}(H(x))\subset H(B_{r/2}(x))$, let $a={\rm Im}\,H(x)\in \mathbb{R}$, recall that $u={\rm Im}\,H$, then
\begin{equation}
    \sup_{B_{2r}(x)}|u|\leq |a|+M(x,2r)\leq C(K)[|a|+m(x,r/2)]\leq C(K)\sup_{B_{r/2}(x)}|u|.
\end{equation}
Combining elliptic estimate, we have
\begin{equation}
    \fint_{B_{2r}(x)}u^{2}\leq C(K,\lambda)\fint_{B_{r}(x)}u^{2}:=4^{\Lambda}\fint_{B_{r}(x)}u^{2},
\end{equation}
i.e., $(3)$ holds. Therefore, we have finish the proof.
\end{proof}

\end{document}